\documentclass[preprint,12pt,authoryear]{elsarticle}

\usepackage[english]{babel}
\usepackage[T1]{fontenc}
\usepackage[utf8]{inputenc}
\usepackage{amsmath,amssymb,amsthm,mathtools,mathrsfs}
\usepackage{graphicx}
\usepackage[caption=false]{subfig}
\usepackage[top=1.25in]{geometry}
\usepackage{enumitem}
\usepackage{booktabs}
\usepackage{xcolor}
\usepackage[normalem]{ulem}
\usepackage{todonotes}
\usepackage[hyphens]{url}
\usepackage{thm-restate}
\usepackage[hidelinks]{hyperref}
\usepackage[capitalise]{cleveref}

\numberwithin{equation}{section}

\newtheorem{theorem}{Theorem}[section]
\newtheorem{lemma}[theorem]{Lemma}
\newtheorem{proposition}[theorem]{Proposition}
\newtheorem{corollary}[theorem]{Corollary}

\theoremstyle{definition}
\newtheorem{assumption}{Assumption}

\newtheorem{algorithm}{Algorithm}

\crefdefaultlabelformat{#2\textup{#1}#3}
\crefname{assumption}{Assumption}{Assumptions}
\Crefname{assumption}{Assumption}{Assumptions}
\Crefrangeformat{assumption}{Assumptions #3\textup{#1}#4--#5\textup{#2}#6}
\crefname{algorithm}{Algorithm}{Algorithms}
\Crefname{algorithm}{Algorithm}{Algorithms}

\journal{Computers \& Chemical Engineering}

\begin{document}

\begin{frontmatter}

\title{Statistical Inference for Scenario-Based Dynamic Optimization under
Uncertainty}

\author[snl]{Aurya Javeed}
\ead{asjavee@sandia.gov}
\address[snl]{Optimization and Uncertainty Quantification, Sandia National Laboratories, Albuquerque, NM 87185, USA}

\author[gt]{Johannes Milz}
\ead{johannes.milz@isye.gatech.edu}
\address[gt]{H.\ Milton Stewart School of Industrial and Systems Engineering, Georgia Institute of Technology, Atlanta, Georgia 30332, USA}

\begin{abstract}
Motivated by batch and semi-batch process operation, we study finite-horizon
open-loop dynamic optimization problems with uncertain parameters. A common
computational approach replaces the expected performance criterion by an average
over finitely many sampled parameter realizations. We develop a statistical
theory for the resulting sample-based optimal value as an estimator of the
population optimal value. The analysis is based on a stability estimate showing
that terminal losses depend Lipschitz continuously on the time-integrated
control, which records the cumulative input delivered up to each time. This
estimate yields a functional central limit theorem for the sample-based
objective and a statistical limit theorem for the corresponding optimal value
error. As a consequence, we obtain confidence intervals for the population
optimal value. When the population optimizer is unique, the limit is Gaussian
and leads to a plug-in confidence interval. When multiple optimal policies may
exist, we use a subsampling confidence interval that does not require
uniqueness. The methodology is illustrated on two fed-batch case studies in
which feed-rate profiles are optimized under parametric uncertainty.
\end{abstract}


\end{frontmatter}

\section{Introduction}

Batch and semi-batch processes are widely used in high-value chemical and
pharmaceutical manufacturing. Operating policies for these processes specify
time-varying inputs, such as feed rates, temperature set points, and dilution
rates, over a finite horizon. Because these policies are chosen before the final
batch outcome is known, run-to-run variation in kinetics, initial composition,
raw materials, and heat transfer can make nominally optimized policies fragile.
This has motivated work on batch process optimization with imperfect models,
parametric uncertainty, robust input policies, and dynamic real-time
optimization; see the survey \cite{Terwiesch1994} and representative
contributions
\cite{Terwiesch1995,Ruppen1995,Terwiesch1998,Paulen2019,Zadravec2025}.

This application context leads us to consider finite-horizon open-loop optimal
control problems in Mayer form, with uncertainty entering through model
parameters. In particular, let \(\xi\) be a random parameter taking values in a
closed set $\Xi \subset \mathbb{R}^d$. We assume the state is driven by a
deterministic control \(u\) and satisfies the control-affine dynamics
\[
\dot x_u(t,\xi)
=
f_0(x_u(t,\xi),\xi)
+
f_1(x_u(t,\xi),\xi)u(t),
\quad
t\in(0,t_f),
\quad
x_u(0,\xi)=x_0(\xi),
\]
where \(f_0 \colon \mathbb{R}^n \times \Xi \to \mathbb{R}^n\), \(f_1 \colon
\mathbb{R}^n \times \Xi \to \mathbb{R}^{n\times s}\), \(t_f>0\) is the fixed, finite final
time, and \(x_0 \colon \Xi \to \mathbb{R}^n\) is a parametric initial value. The
population problem is to choose an admissible open-loop control that minimizes
the expected terminal loss $F \colon \mathbb{R}^n \to \mathbb{R}$ over the
feasible set \(\mathcal U \subset L^2(0,t_f;\mathbb{R}^s)\):
\[
J^*
=
\min_{u \in \mathcal U} J(u),
\quad \text{where} \quad
J(u)
=
\mathbb E[F(x_u(t_f,\xi))].
\]
Here \(L^2(0,t_f;\mathbb{R}^s)\) is the space of \(\mathbb{R}^s\)-valued
square-integrable functions on \([0,t_f]\), and \(\mathbb E\) denotes
expectation with respect to the distribution of \(\xi\).
Throughout, the feasible set $\mathcal U$ is fixed independently of the
scenario sample. It may encode control constraints and deterministic endpoint
constraints. State-path constraints and scenario-dependent endpoint constraints
are not considered here.

In practice, the expectation defining \(J\) is rarely available exactly. Given
independent and identically distributed (i.i.d.) scenarios
\(\xi_1,\ldots,\xi_N\), each distributed as \(\xi\), the sample average
approximation (SAA) \cite{Kleywegt2002,Wei2004} replaces the population problem
by
\[
\widehat J_N^*
=
\min_{u\in\mathcal U}
\widehat{J}_N(u),
\quad \text{where} \quad
\widehat{J}_N(u)
=
\frac1N\sum_{i=1}^N F(x_u(t_f,\xi_i)).
\]
For a fixed realization of the scenario sample, this is a deterministic ensemble
optimal control problem.

The SAA problem is not only a computational surrogate for the population
problem; its optimal value \(\widehat J_N^*\) is also a statistical estimator of
\(J^*\). A different scenario sample may produce a different optimal control and
a different reported optimal value. This sampling error can be comparable to the
reported improvement from using an uncertainty-aware policy. Without uncertainty
quantification, one may therefore overinterpret small differences between
policies, draw misleading conclusions about the value of modeling uncertainty,
or choose a scenario count that is poorly matched to the desired accuracy.  A
scenario-based computation should report not only \(\widehat J_N^*\), but also
how accurate this value is as an estimate of the population optimal value
\(J^*\). The goal of this paper is to provide this statistical error assessment
for optimal values produced by scenario-based dynamic optimization.

\subsection{Contributions}

Our contributions are as follows:
\begin{enumerate}[nosep]
\item We prove a stability estimate for terminal losses with respect to
cumulative input profiles. This estimate allows us to treat the SAA objective as
a stochastic process indexed by admissible controls.

\item We establish a functional central limit theorem for the SAA objective and
derive the limiting distribution of the optimal-value error \(N^{1/2}(\widehat
J_N^*-J^*)\).

\item We construct two asymptotically valid confidence intervals for the
population optimal value \(J^*\). The first is a plug-in interval that applies
when the population optimizer is unique. The second is a subsampling interval
that is computationally more expensive but does not require the population
optimizer to be unique.

\item We illustrate these results on dynamic optimization problems that model
semi-batch process operation under parametric uncertainty. We show how the
confidence intervals quantify the sampling error in the SAA optimal values and
how this error changes with the number of scenarios.
\end{enumerate}

The technical core of our analysis is the stability estimate, which
provides quantitative control of terminal states in terms of cumulative
input profiles.
Once this estimate is established, our functional central
limit theorem follows from standard functional limit theory and entropy arguments applied to the cumulative input profiles; see, e.g.,
\citet{Jain1975,Ledoux1991,Birman1967}. The corresponding limit for the
optimal value then follows from the infinite-dimensional delta method, as is
standard in stochastic optimization; see, e.g.,
\citet{Shapiro2003,Shapiro2021}.

The limiting distribution of \(N^{1/2}(\widehat J_N^*-J^*)\) is, in general, the
infimum of a centered Gaussian process over the set of population optimizers. It
reduces to a Gaussian limit when the population optimizer is unique. Thus, our
functional central limit theorem explains when the plug-in confidence interval
is justified and why the subsampling alternative is useful in the nonunique
case. It also establishes the \(N^{-1/2}\) scaling needed to interpret the
tradeoff between computational cost and statistical accuracy.

\subsection{Literature review}

For finite-dimensional stochastic optimization, statistical limit theorems are
classical; see \cite{Dupacova1988,Shapiro1991,Shapiro2003,Shapiro2021}. 
For optimal control problems over
infinite-dimensional decision spaces, the statistical theory is more
specialized. \cite{Roemisch2024} prove central limit theorems for
linear-quadratic problems governed by uncertain partial differential equations.
\citet{Milz2026} establish central limit theorems for SAA
value functions and optimal values in discrete-time, finite-horizon
stochastic optimal control with closed-loop policies, 
whereas the present work considers
continuous-time open-loop control under parametric uncertainty.
Consistency of the SAA formulation for optimal control under uncertainty, that
is, convergence to the population problem as the number of scenarios increases,
is studied in \cite{Phelps2016,Scagliotti2023,Milz2024a,Milz2025}. More
recently, \cite{Melnikov2026} derive sample complexity bounds for optimal
control under uncertainty, but their framework does not cover the Mayer costs
considered here.
In our setting, the decision space is also infinite-dimensional,
so a functional limit theorem over the set of feasible cumulative input profiles is needed
before standard delta theorem arguments can be
applied; see, e.g., \citet{Shapiro2003,Shapiro2021}.

\citet{Liu1999} study input convergence conditions that ensure convergence of
trajectories for control-affine systems, including conditions formulated in
terms of time-integrated controls. Classical weak and weak-* continuity
arguments are also well established in optimal control; see, e.g.,
\citet{Gamkrelidze1978,Warga1972}. For our functional limit theorem, we require
the stronger quantitative stability estimate developed here, which bounds terminal-state
differences directly by the uniform distance between cumulative input profiles.

The SAA formulation can also be viewed through the lens of M-estimation
\cite{Vaart2023}: the empirical criterion $\widehat J_N$ is minimized over the
admissible decision class $\mathcal U$. In chemical engineering, M-estimators
have also been used in regression analysis and data reconciliation, for example
by \cite{Cunha2021,Menezes2021,Llanos2023}.

\section{Problem formulation}

This section collects assumptions needed to obtain statistical limit theorems
for the problem class introduced above. We state assumptions in a form tailored
to the computational setting. Standard measurability conditions are assumed
throughout, so that all expectations and sample averages are well defined. These
technical details are not emphasized, since the focus here is on the optimal
control model and the resulting sampling error.

The first assumption is a compactness condition on the admissible controls. It
is the infinite-dimensional analogue of optimizing over a closed and bounded set
of decision variables.
We recall that the final time $t_f >0$ is a fixed, finite constant.

\begin{assumption}
\label{ass:weakly-compact-feasible-set}
The feasible set $\mathcal{U} \subset L^2(0, t_f; \mathbb{R}^s)$ is nonempty,
bounded, and weakly closed in $L^2(0, t_f; \mathbb{R}^s)$.
\end{assumption}

A standard example satisfying \Cref{ass:weakly-compact-feasible-set} is the
box-constrained set
\[
\mathcal U
=
\{u\in L^2(0,t_f;\mathbb{R}^s):
\underline u\le u(t)\le \overline u
\ \text{for almost every }t\in[0,t_f]\},
\]
where \(\underline u,\overline u\in\mathbb{R}^s\) are finite componentwise
bounds with \(\underline u\le \overline u\).

The next assumption requires all trajectories generated by feasible controls
and parameter values to remain in a common compact, convex region
over the fixed, finite horizon.

\begin{assumption}[Compact reachable set]\label{ass:reachable} There exists a
compact, convex set $\mathcal{X}\subset\mathbb R^n$ such that, for every
$u\in\mathcal{U}$, every $t\in[0,t_f]$, and every $\xi\in\Xi$, $x_u(t,\xi)\in
\mathcal{X}$.
\end{assumption}

The next assumption requires the process model to be smooth on the reachable
operating region. This is enough to obtain the trajectory continuity estimates
used in the statistical analysis.

\begin{assumption}[Regular sample space, initial conditions, and dynamics]
\label{ass:c1}
\textup{(i)}
The uncertainty set \(\Xi\subset\mathbb R^d\) is compact. \textup{(ii)} The
parameterized initial value $x_0 \colon \Xi \to \mathbb{R}^n$ is measurable.
\textup{(iii)} The functions \(f_0\) and \(f_1\) are continuous on \(\mathcal
X\times\Xi\). Moreover, for each \(\xi\in\Xi\), the maps \(f_0(\cdot,\xi)\) and
\(f_1(\cdot,\xi)\) are differentiable in a neighborhood of \(\mathcal X\), and
their state derivatives \(D_x f_0\) and \(D_x f_1\)  are continuous on
\(\mathcal X\times\Xi\).
\end{assumption}

Finally, the terminal objective must not amplify small endpoint errors into
large objective errors.

\begin{assumption}
\label{ass:F-Lipschitz}
The terminal loss \(F\) is Lipschitz continuous on \(\mathcal X\).
\end{assumption}

Together, these assumptions provide the compactness and regularity needed
for the statistical analysis. Weak compactness of \(\mathcal U\) yields
compactness of the cumulative input profiles, while compact reachability and
smooth dynamics provide uniform trajectory bounds used in the stability
estimate. The Lipschitz condition on \(F\) then transfers this stability to the
terminal loss. Because the horizon is fixed and finite, compact reachability is only a
finite-horizon requirement and does not, in particular, exclude integrator
states.

\section{Statistical limit theorems for optimal values}

The SAA problem produces both an optimal control and an optimal value. In
applications, optimal values are often used to compare operating policies,
assess the benefit of robustness, or decide whether additional scenarios are
needed. For these purposes, the point estimate \(\widehat J_N^*\) is not enough.
One also needs to quantify the sampling error caused by replacing the population
objective \(J\) with the sample objective \(\widehat{J}_N\).

The result below gives the scale and limiting form of this sampling error. In
particular, the SAA optimal value converges at the standard Monte Carlo rate
\(N^{-1/2}\). Thus, increasing the number of scenarios by a factor of four
reduces the statistical error in the reported optimal value by about a factor of
two.

Denote the set of population optimal controls by
\[
\mathcal U^*
=
\{u\in\mathcal U:J(u)=J^*\}.
\]
The following  corollary of \Cref{cor:control_indexed_clt_CK} in
 \Cref{app:proof-optimal-value-clt}
describes the limiting distribution of the SAA optimal
value error. We write \(\rightsquigarrow\) for convergence in distribution and
use \(\mathcal N(0,\sigma^2)\) for a normal distribution with mean zero and
variance \(\sigma^2\). The limit is expressed in terms of a centered Gaussian
process \(Z=\{Z(u):u\in\mathcal U\}\), meaning that for any finite collection of
controls \(u_1,\ldots,u_k \in \mathcal{U}\), the random vector
\((Z(u_1),\ldots,Z(u_k))\) has a mean-zero multivariate normal distribution.

\begin{corollary}[Statistical limit theorem for optimal values]
\label{thm:optimal_value_clt}
If \Cref{ass:weakly-compact-feasible-set,ass:reachable,ass:c1,ass:F-Lipschitz}
hold, then
\begin{align}
\label{eq:optimal_value_limit'}
N^{1/2}(\widehat{J}_N^*-J^*)
\rightsquigarrow
\inf_{u\in\mathcal U^*} Z(u),
\end{align}
where \(Z\) is a centered Gaussian process on $\mathcal{U}$ with covariance
\[
\operatorname{Cov}(Z(u),Z(v))
=
\operatorname{Cov}
\bigl(
F(x_u(t_f,\xi)),F(x_v(t_f,\xi))
\bigr),
\quad u, v \in \mathcal{U}.
\]
In particular, if the population problem has a unique optimal control \(u^*\),
so that \(\mathcal U^*=\{u^*\}\), then
\[
N^{1/2}(\widehat{J}_N^*-J^*)
\rightsquigarrow
\mathcal{N}
\bigl(0,\operatorname{Var}(F(x_{u^*}(t_f,\xi)))\bigr).
\]
\end{corollary}

The proof of this result is deferred to \Cref{app:proof-optimal-value-clt}. As
noted above, the limit in \eqref{eq:optimal_value_limit'} gives the usual Monte
Carlo convergence rate for the SAA optimal value. The reported value \(\widehat
J_N^*\) should therefore be interpreted as an estimate of \(J^*\) with sampling
error of order \(N^{-1/2}\). This rate is independent of the dimension of the
random parameter vector $\xi$, but the constant in the error distribution
depends on the variability of the terminal loss at optimal and near-optimal
controls.

The form of the limiting distribution has direct consequences for statistical
error assessment. When the population optimizer is unique, the limit is
Gaussian, so the sampling error can be estimated by a standard variance estimate
at the computed optimizer. When several controls are population-optimal, the
limit is the lower envelope of a Gaussian process and is non-Gaussian in
general. This distinction is the basis for the two asymptotically valid
confidence intervals introduced next.

\section{Estimating the reliability of the computed solution}
\label{sec:confidence-intervals}

The SAA optimal value \(\widehat J_N^*\) is a computable estimate of the
population optimal value \(J^*\). In practice, however, the reported value is
only useful insofar as its sampling error is understood. This section describes
two asymptotically valid confidence intervals for assessing that sampling error.
The proofs of the results in this section are deferred to
\Cref{app:proof-confidence-intervals}.

The first interval is simple and inexpensive to compute, but it relies on
uniqueness of the population optimizer. Under this condition, the limiting
distribution of the SAA optimal value error is Gaussian. This gives a normal
confidence interval for \(J^*\), with the variance estimated from the terminal
losses evaluated at the computed SAA optimizer. When the population optimizer is
not unique, this normal approximation is generally no longer valid. We then use
subsampling: smaller SAA problems are repeatedly solved on subsets of the
original scenarios, and the resulting empirical distribution is used to
approximate the sampling error of the full SAA optimal value.

\subsection{Plug-in confidence intervals under a unique optimizer}
\label{sec:gaussian_ci_unique_optimizer}

We first consider the case in which the population problem has a unique
optimizer \(u^*\). This is the most favorable case: the random fluctuation of
the SAA optimal value is the same, to first order, as the fluctuation of the
sample average at \(u^*\). Since \(u^*\) is unknown, we replace it by an SAA
optimizer.

\begin{algorithm}[Plug-in confidence interval under a unique optimizer]
\label{alg:plugin_ci_unique_optimizer}
Fix a confidence level \(1-\beta\), with \(\beta\in(0,1)\), and choose a sample
size \(N\).

\begin{enumerate}[nosep]
\item  Generate i.i.d.\ scenarios \(\xi_1,\ldots,\xi_N\), each distributed as
\(\xi\).

\item Solve the SAA  problem to obtain an optimizer \(u_N^*\) and the optimal
value \(\widehat{J}_N^*\).

\item Compute the variance estimator
\[
\widehat{\sigma}_N^2
=
\frac{1}{N}
\sum_{i=1}^N
\bigl(F(x_{u_N^*}(t_f,\xi_i))-\widehat{J}_N^*\bigr)^2 .
\]

\item Let \(z_{1-\beta/2}\) be the \((1-\beta/2)\)-quantile of the standard
normal distribution, and report
\[
\mathcal{I}_N^{\mathrm{in}}(\beta)
=
\biggl[
\widehat{J}_N^*
-
z_{1-\beta/2}\frac{\widehat{\sigma}_N}{N^{1/2}},
\quad
\widehat{J}_N^*
+
z_{1-\beta/2}\frac{\widehat{\sigma}_N}{N^{1/2}}
\biggr].
\]
\end{enumerate}
\end{algorithm}

The next result states that the interval computed by
\Cref{alg:plugin_ci_unique_optimizer} has asymptotic coverage at the prescribed
level.

\begin{theorem}[Validity of the plug-in confidence interval]
\label{thm:gaussian_ci_unique_optimizer}
Suppose that
\Cref{ass:weakly-compact-feasible-set,ass:reachable,ass:c1,ass:F-Lipschitz} hold
and that the population problem has a unique optimal control \(u^*\). If the
variance of \(F(x_{u^*}(t_f,\xi))\) is positive, then
\Cref{alg:plugin_ci_unique_optimizer} provides an asymptotically valid
\(1-\beta\) confidence interval \(\mathcal I_N^{\mathrm{in}}(\beta)\) for
\(J^*\), in the sense that
\[
\lim_{N\to\infty}
\mathrm{Prob}
\bigl\{
J^* \in \mathcal I_N^{\mathrm{in}}(\beta)
\bigr\}
=
1-\beta .
\]
\end{theorem}

\subsection{A subsampling confidence interval for nonunique optimizers}
\label{sec:subsampling_ci}

If the population problem has several optimizers, the limiting distribution of
\(N^{1/2}(\widehat J_N^*-J^*)\) is typically non-Gaussian. The source of the
difficulty is that the SAA optimizer may move among several nearly optimal
controls, and the optimal value is then governed by the lower envelope of
several random fluctuations.

We therefore use subsampling, which estimates the limiting distribution of
\(N^{1/2}(\widehat J_N^*-J^*)\) by resolving smaller SAA problems on subsamples
of the original scenarios. This construction follows the subsampling principle
of \cite{Politis1994}, as applied to stochastic optimization in
\cite{Eichhorn2007}.

\begin{algorithm}[Subsampling confidence interval]
\label{alg:subsampling_ci}
Fix a confidence level \(1-\beta\), a sample size \(N\), a subsample size
\(b_N<N\), and a number of subsamples \(m_N\).

\begin{enumerate}[nosep]
\item  Generate i.i.d.\ scenarios \(\xi_1,\ldots,\xi_N\), each distributed as
\(\xi\).

\item Solve the full SAA optimal control problem to obtain the optimal value
\(\widehat{J}_N^*\).

\item Independently for \(r=1,\ldots,m_N\), draw a subset
$
I_r\subset\{1,\ldots,N\}
$
with $b_N$ elements uniformly without replacement.

\item For  \(r=1,\ldots,m_N\), solve the SAA problem using only the scenarios
indexed by \(I_r\), and denote its optimal value by \(\widehat J_{I_r}^*\).

\item Compute the subsampling statistics
$
\Delta_r
=
b_N^{1/2}(\widehat J_{I_r}^*-\widehat{J}_N^*)
$,
$
r=1,\ldots,m_N$.

\item Sort the values \(\Delta_1,\ldots,\Delta_{m_N}\) in increasing order. Let
\(q_{N,b_N,m_N}(p)\) be the lower empirical \(p\)-quantile, that is, the value
with rank \(\lceil m_Np\rceil\) in this sorted list. Report
\[
\mathcal I_{N,b_N,m_N}^{\mathrm{sub}}(\beta)
=
\biggl[
\widehat{J}_N^*
-
q_{N,b_N,m_N}(1-\beta/2)\frac{1}{N^{1/2}},
\quad
\widehat{J}_N^*
-
q_{N,b_N,m_N}(\beta/2)\frac{1}{N^{1/2}}
\biggr].
\]
\end{enumerate}
\end{algorithm}

Unlike the plug-in confidence interval $\mathcal{I}_N^{\mathrm{in}}(\beta)$
above, the subsampling confidence interval \(\mathcal
I_{N,b_N,m_N}^{\mathrm{sub}}(\beta)\) is not generally symmetric about
\(\widehat{J}_N^*\). Moreover, \(\widehat J_N^*\) need not lie in the
subsampling confidence interval.

The next result states that the interval computed by \Cref{alg:subsampling_ci}
has asymptotic coverage at the prescribed level.

\begin{theorem}[Validity of the subsampling confidence interval]
\label{thm:subsampling_ci}
Let \(\beta\in(0,1)\).
Suppose the hypotheses of \Cref{thm:optimal_value_clt} hold, 
and the limiting random variable from \eqref{eq:optimal_value_limit'}
has no atom at either its \((\beta/2)\)-quantile or its \((1-\beta/2)\)-quantile.
If
$
b_N\to\infty
$,
$b_N/N \to 0$, and $m_N \to \infty$, then, 
\Cref{alg:subsampling_ci} provides an asymptotically valid \((1-\beta)\)
confidence interval for \(J^*\), in the sense that
\[
\lim_{N\to\infty}
\mathrm{Prob}
\big\{
J^*
\in
\mathcal I_{N,b_N,m_N}^{\mathrm{sub}}(\beta)
\big\}
=
1-\beta .
\]
\end{theorem}

\section{Numerical illustrations}

The numerical examples focus on fed-batch process optimization under parametric
uncertainty. In both cases, the feed rate is optimized to improve an endpoint
production objective, while uncertainty enters through kinetic parameters;
in the first example, it also enters through the initial condition. The
examples therefore represent the type of scenario-based dynamic optimization
problem encountered in batch and semi-batch operation: one computes an open-loop
policy and an associated expected performance value from a finite scenario
sample. The purpose of this section is to quantify the sampling error in that
reported optimal value using the confidence intervals above.

We first describe the numerical setting used in both examples. The nominal
control problem is given by
\[
\min_{u \in \mathcal{U}}\, F(x_u(t_f, \mathbb{E}[\xi])),
\]
which replaces the uncertain parameters by their expected value. When
visualizing controls, we include the nominal control as a deterministic
baseline, so that the effect of accounting for parametric uncertainty can be
seen directly. We use single shooting with a fixed-step explicit fourth-order
Runge--Kutta method. Controls are piecewise constant on a uniform grid of
\([0,t_f]\) with \(q\) subintervals. Finer grids are used for visualizing
nominal and SAA controls. The statistical results require repeated SAA solves,
so they use a coarser fixed grid to keep the computation reproducible and
affordable. Computations use the Python package EnsembleControl \cite{Milz2024},
built on CasADi \cite{Andersson2019} and IPOPT \cite{Waechter2006}. Our code and
its output are archived at \cite{Javeed2026}. All simulations were run on a
laptop with an Apple M4 chip and 16 gigabytes of RAM.

\subsection{Fed-batch reactor under uncertainty}
\label{sec:luus-reactor}

The first example is based on the fed-batch reactor control problem of
\citet{Luus1992}; the underlying fed-batch model is due to
\citet{Park1988}.
The original problem is deterministic;
here we introduce parametric uncertainty in the numerical constants
appearing in the dynamics and initial condition, and use the example to
illustrate uncertainty-aware dynamic optimization. The state is
\(x(t,\xi)\in\mathbb R^5\), the scalar control is \(u(t)\), and the final time
is fixed at \(t_f=15\) hours.
The states \(x_1\), \(x_2\), \(x_3\), \(x_4\), and \(x_5\)
represent the secreted-protein concentration, total-protein concentration,
cell density, glucose concentration, and culture volume, respectively, while
\(u\) is the glucose feed flow rate.
For a parameter \(\xi \in \mathbb{R}^{11}\), the dynamics are
\[
\begin{aligned}
\dot x_1(t,\xi)
&=
g_1(x_4(t,\xi),\xi)
\bigl(x_2(t,\xi)-x_1(t,\xi)\bigr)
-
\frac{u(t)}{x_5(t,\xi)}x_1(t,\xi),
\\
\dot x_2(t,\xi)
&=
g_2(x_4(t,\xi),\xi)x_3(t,\xi)
-
\frac{u(t)}{x_5(t,\xi)}x_2(t,\xi),
\\
\dot x_3(t,\xi)
&=
g_3(x_4(t,\xi),\xi)x_3(t,\xi)
-
\frac{u(t)}{x_5(t,\xi)}x_3(t,\xi),
\\
\dot x_4(t,\xi)
&=
\xi_8\, g_3(x_4(t,\xi),\xi)x_3(t,\xi)
+
\frac{u(t)}{x_5(t,\xi)}
\bigl(20-x_4(t,\xi)\bigr),
\\
\dot x_5(t,\xi)
&=
u(t).
\end{aligned}
\]
The kinetic functions are
\[
\begin{aligned}
g_3(x_4,\xi)
&=
\frac{\xi_5 x_4}{(x_4+\xi_6)(x_4+\xi_7)},
\quad
g_2(x_4,\xi)
=
\frac{x_4\exp(\xi_3 x_4)}{\xi_4+x_4},
\\
g_1(x_4,\xi)
&=
\frac{\xi_1 g_3(x_4,\xi)}{\xi_2+g_3(x_4,\xi)}.
\end{aligned}
\]
The functions $g_1$, $g_2$, and $g_3$ are kinetic functions governing
protein secretion, protein expression, and cell growth, respectively. Because
these functions depend differently on the glucose concentration $x_4$,
different glucose levels favor cell growth, protein expression, and secretion
at different stages of the batch.
The parameterized initial condition equals
$
x_0(\xi)
=
(0,0,\xi_9,\xi_{10},\xi_{11})
$,
and we use the control bounds $0\le u(t)\le 2$, $t \in [0, t_f]$. The
performance index is
\[
F(x(t_f,\xi))
=
-x_1(t_f,\xi)x_5(t_f,\xi),
\]
so minimizing \(F\) maximizes the final amount of secreted protein. The nominal
parameter values are
\[
\bar\xi
=
(4.75,\;0.12,\;-5,\;0.1,\;21.87,\;0.4,\;62.5,\;-7.3,\;1,\;5,\;1),
\]
where the first eight components, \((\xi_1,\dots,\xi_8)\), enter the dynamics
and the last three, \((\xi_9,\xi_{10},\xi_{11})\), enter the initial condition.
We model parametric uncertainty by relative perturbations of the nominal
parameters. In particular,
\[
\xi_j
=
(1+r\eta_j)\bar\xi_j,
\quad j = 1, \ldots, 11,
\quad
r=0.04,
\]
where the \(\eta_j\) are independent and uniformly distributed over $[-1,1]$.

We outline why this example satisfies
\Cref{ass:F-Lipschitz,ass:c1,ass:reachable,ass:weakly-compact-feasible-set}.
The feasible set $\mathcal U=\{u\in L^2(0,t_f):
0\le u(t)\le2\ \text{for almost every }t\in[0,t_f]\}$ satisfies
\Cref{ass:weakly-compact-feasible-set}, and
the uncertainty set 
$\Xi=\{\xi\in\mathbb R^{11}:\xi_j=(1+r\eta_j)\bar\xi_j,
\ \eta_j\in[-1,1],\ j=1,\ldots,11\}$ is compact.
The feasible feed rates and initial volumes keep the culture volume
$x_5(t,\xi)$ between $1/2$ and $32$. Nagumo-type invariance criteria give
nonnegative concentrations and $x_4(t,\xi)\le20$. Moreover, Gronwall's inequality
provides a finite upper bound $M$ for $x_i(t,\xi)$, $i=1,2,3$. Therefore,
$x(t,\xi)\in\mathcal X=[0,M]^3\times[0,20]\times[1/2,32]$
for all $u\in\mathcal U$, $t\in[0,t_f]$, and $\xi\in\Xi$, as required by \Cref{ass:reachable}.
For \Cref{ass:c1}, $\Xi$ is compact, the initial conditions $x_0$ are continuous
in $\xi$, and positive volume and denominators in $g_1,g_2,g_3$ on
$\mathcal X\times\Xi$ give smooth dynamics. The terminal loss $F$ is smooth,
and hence Lipschitz on the compact, convex set $\mathcal X$, as required by
\Cref{ass:F-Lipschitz}.

\Cref{fig:fed-batch-controls} depicts the postprocessed nominal control and a
postprocessed SAA solution, with $q$ and the sample size indicated in the
legend. Because the dynamics are affine in the control, the control problem is
potentially singular, and the direct controls chatter. We therefore postprocess the direct
controls from the first-order optimality conditions: integrating the state and
adjoint of each scenario and forming the ensemble switching function, we hold
the control at its bounds on the bang arcs and replace it by the reconstructed
singular control on the singular arc. The nominal control is obtained the same
way from the single nominal parameter.

\begin{figure}[t]
\centering
\subfloat[Postprocessed  nominal solution]{%
  \includegraphics[width=0.47\linewidth]{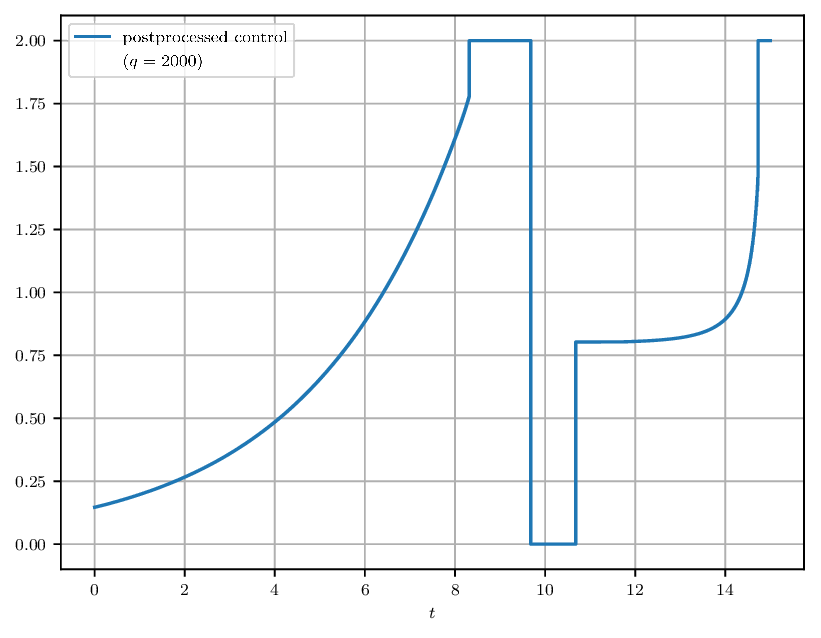}%
}\hfill \subfloat[Postprocessed SAA solution]{%
  \includegraphics[width=0.47\linewidth]{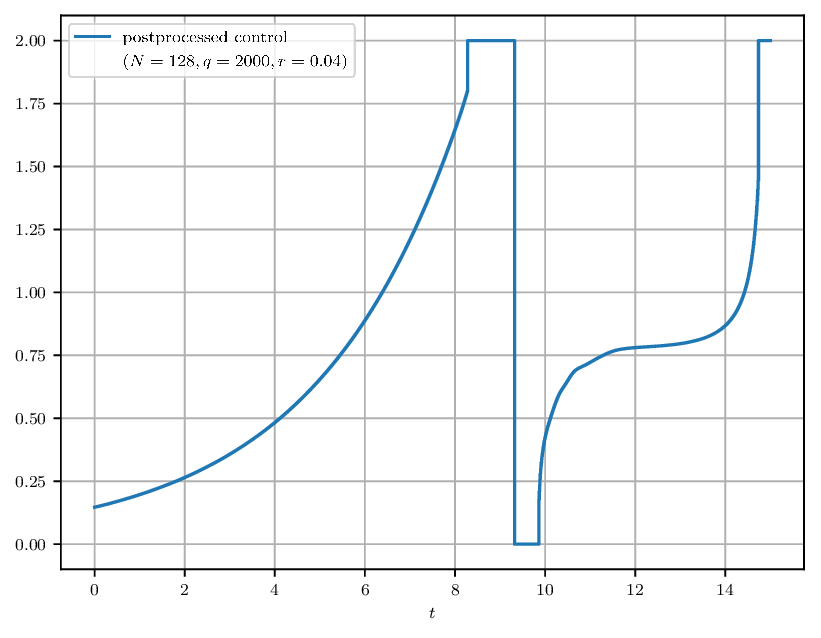}%
} \caption{For the fed-batch reactor example, postprocessed nominal and SAA
optimal solutions.}
\label{fig:fed-batch-controls}
\end{figure}

\Cref{fig:fed-batch-clt} illustrates the finite-sample behavior of the scaled
SAA optimal-value error \(N^{1/2}(\widehat J_N^* - J^*)\). Since \(J^*\) is
unknown, we replace it by an independent reference value \(\widehat
J_{N_\mathrm{ref}}^*\), computed from a larger sample. The comparable spread of
the scaled errors across the $N=32$ and 64 sample sizes is consistent with the
\(N^{-1/2}\) optimal-value error predicted by \Cref{thm:optimal_value_clt}. The
figure also reports empirical estimates of \(\mathbb E[\widehat J_N^*]\). As
expected for a minimization problem, these means approach the population optimal
value from below \cite{Mak1999,Shapiro2021}, reflecting the well-known
optimistic bias of SAA optimal values.

\begin{figure}[t]
\centering
\subfloat[Sample size $N = 32$]{%
  \includegraphics[width=0.33\linewidth]{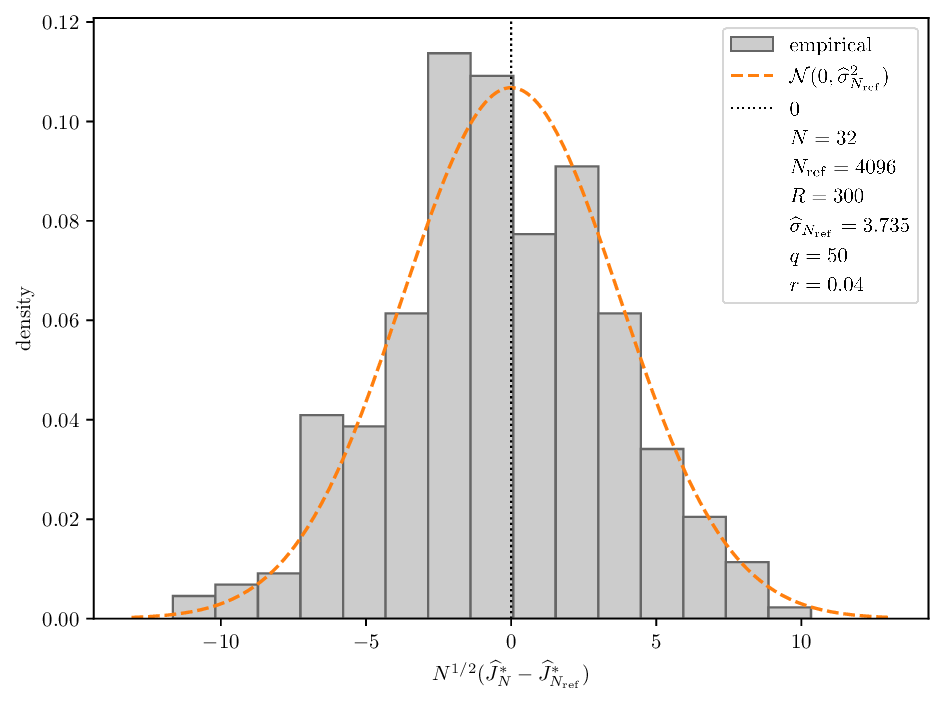}%
}\hfill \subfloat[Sample size $N = 64$]{%
  \includegraphics[width=0.33\linewidth]{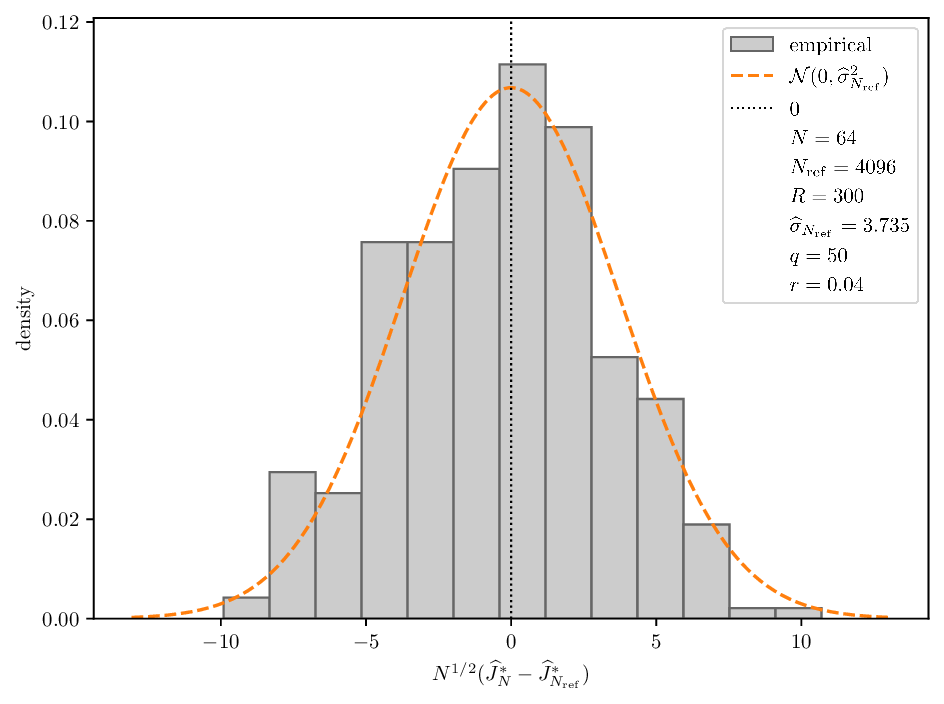}%
}\hfill \subfloat[Optimization bias]{%
  \includegraphics[width=0.33\linewidth]{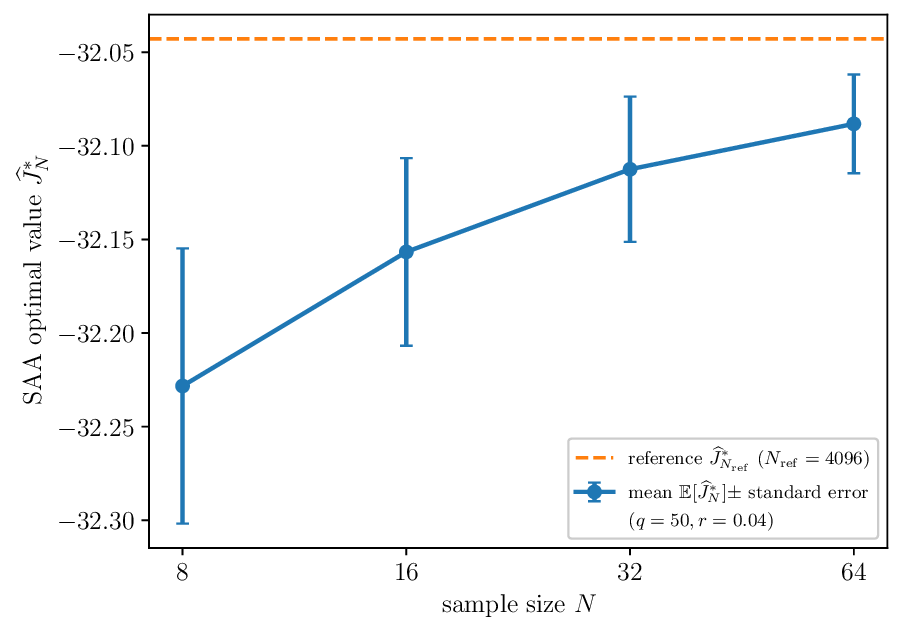}%
} \caption{%
For the fed-batch reactor example, finite-sample behavior of the SAA optimal
value. Panels~(a) and (b) show histograms of \(N^{1/2}(\widehat J_N^*-\widehat
J_{N_\mathrm{ref}}^*)\) over \(R\) i.i.d.\ replications for the indicated sample
size \(N\). Panel~(c) reports the empirical mean of \(\widehat J_N^*\) across
replications. Here \(\widehat J_N^*\) is the SAA optimal value computed from
\(N\) i.i.d.\ parameter scenarios, and \(\widehat J_{N_\mathrm{ref}}^*\) is an
independent reference SAA optimal value computed with \(N_{\mathrm{ref}}\)
scenarios. }
\label{fig:fed-batch-clt}
\end{figure}

\Cref{fig:fed-batch-confidence-intervals} compares the two proposed confidence
intervals. For the subsampling confidence intervals, we use $b_N = \lfloor
N^{6/7} \rfloor$ and $m_N = 5 N$. For each sample size \(N\), we solve the SAA
problem and report the resulting optimal value \(\widehat J_N^*\) together with
a \(1-\beta\) confidence interval for the population optimal value \(J^*\). The
plug-in interval uses the normal approximation, whereas the subsampling interval
uses the empirical distribution of smaller subsample optimal values.
We use common random numbers across sample sizes, taking the first $N$
scenarios from a single sequence of 64 i.i.d.\ scenarios. The resulting samples
are therefore nested. Hence, the intervals are dependent across sample sizes. For the
displayed realization, the interval widths
decrease as \(N\) increases, illustrating the reduction in uncertainty in
\(\widehat J_N^*\) as the scenario budget grows.

The plug-in interval is computationally inexpensive, but its validity requires
uniqueness of the population optimizer. Since uniqueness is difficult to verify
for nonlinear dynamic optimization problems, we interpret the plug-in intervals
as diagnostics based on the computed SAA optimizers. The subsampling intervals
are less structurally restrictive because they remain applicable without
uniqueness of the population optimizer. In finite samples, either interval can
be wider; in \Cref{fig:fed-batch-confidence-intervals}, the plug-in intervals
are wider than the subsampling intervals. The figure shows how the estimated
sampling error in \(\widehat J_N^*\) decreases with the scenario budget \(N\).

\begin{figure}[t]
\centering
\subfloat[Plug-in confidence interval]{%
  \includegraphics[width=0.47\linewidth]{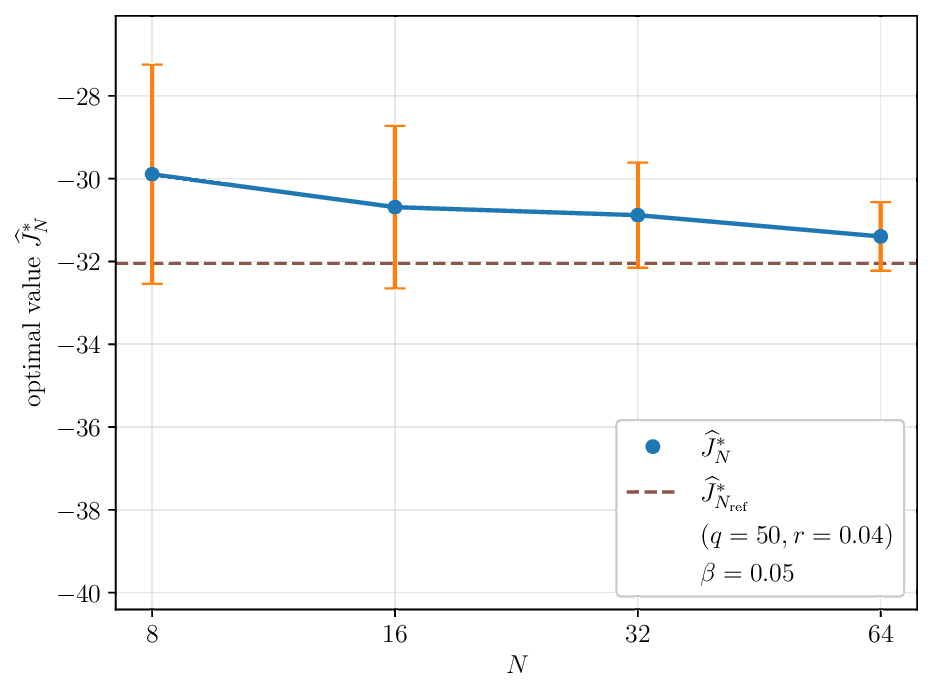}%
}\hfill \subfloat[Subsampling confidence interval]{%
  \includegraphics[width=0.47\linewidth]{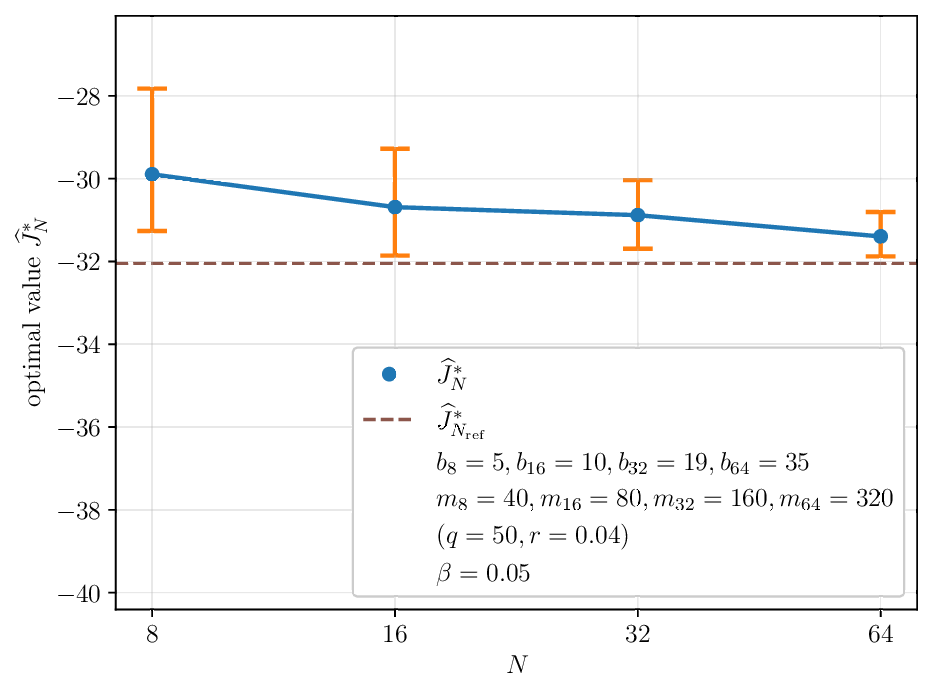}%
} \caption{%
For the fed-batch reactor example, confidence intervals for the population
optimal value $J^*$ versus the sample size $N$. Each panel shows the SAA optimal
value $\widehat J_N^*$ with its $1-\beta$ confidence interval. Panel (a) uses
the plug-in normal interval; panel (b) uses the subsampling interval.
The dashed line marks the independent SAA reference value
$\widehat J_{N_{\mathrm{ref}}}^*$ with $N_{\mathrm{ref}}=4096$.}
\label{fig:fed-batch-confidence-intervals}
\end{figure}

We examine how the subsample size affects the width of the subsampling
confidence interval. For several choices of the subsample size $b_N$, we repeat 
the subsampling calculation 30 times and report the average interval width.
\Cref{fig:fed-batch-subsampling-sensitivity} shows the results for \(N=32\) and \(N=64\).
For both sample sizes, the average interval width decreases as the
subsample size increases over the displayed range.

\begin{figure}[t]
\centering
\subfloat[Sample size $N=32$]{%
  \includegraphics[width=0.47\linewidth]%
  {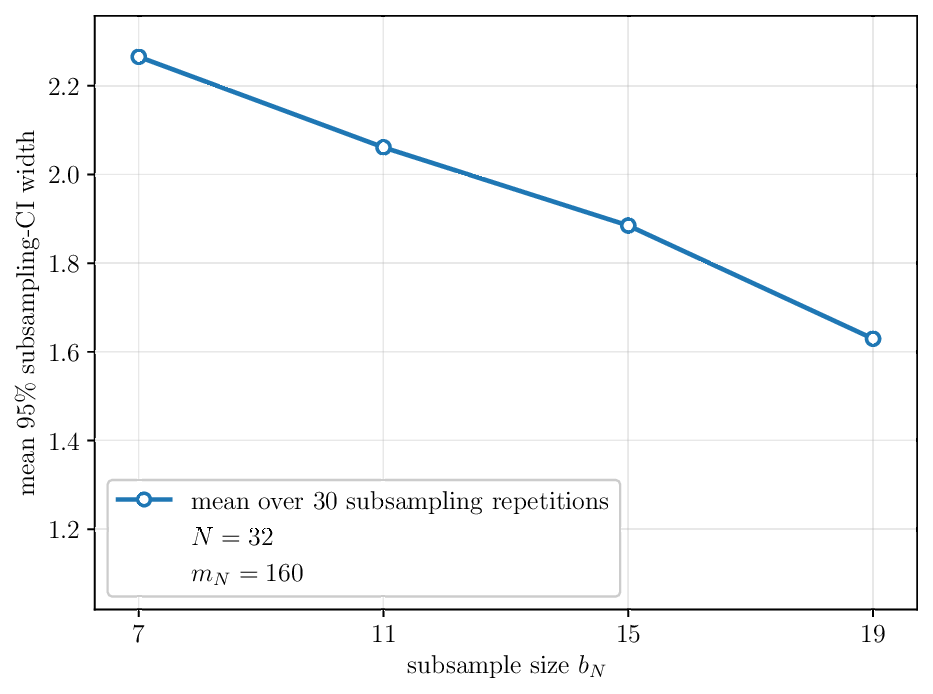}%
}\hfill
\subfloat[Sample size $N=64$]{%
  \includegraphics[width=0.47\linewidth]%
  {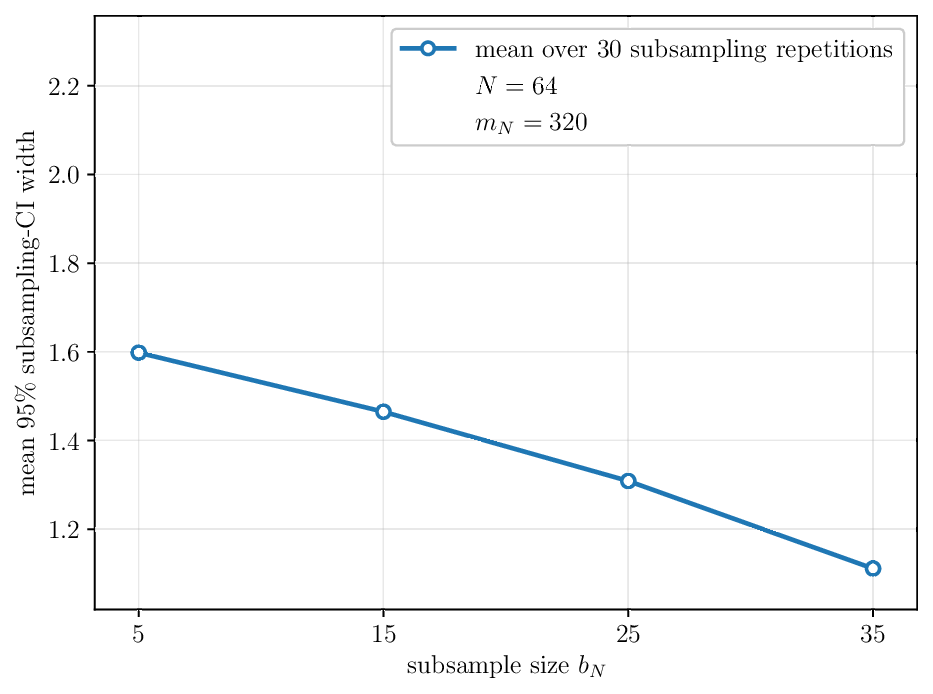}%
}
\caption{%
Mean width of the nominal \(95\%\) subsampling confidence interval as a
function of \(b_N\) for the fed-batch reactor example. Each value is averaged
over \(30\) subsampling repetitions with \(m_N=5N\). The panels share the same
vertical scale.}
\label{fig:fed-batch-subsampling-sensitivity}
\end{figure}

Our theoretical confidence interval guarantees are asymptotic. We next assess
the finite-sample coverage of the plug-in interval in this example. We do not
include the analogous study for the subsampling interval, because it would
require substantially more repeated SAA solves. 
To report coverage in a reliable
way, we use the following simulation-based lower-bound estimator from
\citet{BenTal2009,Nemirovski2007}. Let \(p\in[0,1]\), and let \(\zeta\) be a
Bernoulli random variable with success probability \(p\). For i.i.d.\ copies
\(\zeta_1,\ldots,\zeta_R\), let \(L=\sum_{r=1}^R \zeta_r\) be the number of
successes. As in Section~10.2.1 of \citet{BenTal2009} and p.~979 of
\citet{Nemirovski2007}, we define
\[
\widehat p_{R,\delta}(L)
=
\min\bigg\{
q\in[0,1]:
\sum_{\ell=L}^R {R\choose \ell} q^\ell(1-q)^{R-\ell}
\ge \delta
\bigg\}.
\]
Lemma~10.2.1 of \citet{BenTal2009} implies that, with probability at least
\(1-\delta\), \(\widehat p_{R,\delta}(L)\) is a lower bound on \(p\). In
contrast, the standard Monte Carlo estimator \(L/R\) is  unbiased, but is not a
certified lower bound. We use \(\delta=10^{-6}\).

We apply this estimator to the plug-in interval as follows. Since the exact
population optimal value \(J^*\) is not available analytically, we compute an
independent reference value \(\widehat J_{N_\mathrm{ref}}^*\) and assess coverage
relative to this fixed value. For each independent replication \(r\), we
construct the plug-in confidence interval and set
\[
\zeta_r
=
\mathbf{1}\{\widehat J_{N_\mathrm{ref}}^*
\text{ belongs to the plug-in interval in replication } r\},
\]
where $\mathbf{1}$ is the indicator function. Thus, \(p\) is the probability
that the plug-in interval covers \(\widehat J_{N_\mathrm{ref}}^*\) in one independent
replication. The estimator \(\widehat p_{R,\delta}(L)\), with \(L=\sum_{r=1}^R
\zeta_r\), gives a certified lower bound on this coverage probability.

The coverage results are shown in \Cref{table:fed-batch-coverage}. The Monte
Carlo coverage \(L/R\) is below the nominal level for all tested sample sizes
and confidence levels, although the coverage generally improves as \(N\)
increases. The reliability-adjusted lower bounds \(\widehat p_{R,\delta}(L)\)
are also below
the corresponding nominal confidence levels. Thus, for this example and the
tested sample sizes, the plug-in confidence intervals fall short of their
nominal coverage, particularly at smaller sample sizes.
The observed undercoverage may reflect optimistic bias of the
SAA optimal value, error in the normal approximation, and sampling
error in the reference optimal value. The Student-$t$ modification
in \Cref{app:student-t-coverage} improves coverage, particularly for
small $N$, without isolating these effects. The observed undercoverage does not
contradict the asymptotic coverage guarantee 
provided by \Cref{thm:gaussian_ci_unique_optimizer}.

\begin{table}[t]
  \centering
  \caption{For the fed-batch reactor example, estimated coverage of the
  plug-in confidence interval for the population optimal value. The population
  optimum is proxied by the SAA optimal value on an independent reference sample of size
  $N_{\mathrm{ref}} = 4096$. For each training size $N$ and nominal level
  $1-\beta$, $L/R$ is the empirical coverage over $R=10000$ replications and
  $\widehat p_{R,\delta}(L)$ is the $(1-\delta)$ lower confidence bound on 
the coverage probability relative to the fixed reference value.}
  \begin{tabular}{rcccccc}
    \toprule
     & \multicolumn{2}{c}{$1-\beta = 0.90$} & \multicolumn{2}{c}{$1-\beta = 0.95$} & \multicolumn{2}{c}{$1-\beta = 0.99$} \\
    \cmidrule(lr){2-3}\cmidrule(lr){4-5}\cmidrule(lr){6-7}
    $N$ & $L/R$ & $\widehat p_{R,10^{-6}}(L)$ & $L/R$ & $\widehat p_{R,10^{-6}}(L)$ & $L/R$ & $\widehat p_{R,10^{-6}}(L)$ \\
    \midrule
    8 & 0.839 & 0.821 & 0.895 & 0.880 & 0.953 & 0.943 \\
    16 & 0.872 & 0.856 & 0.922 & 0.909 & 0.971 & 0.962 \\
    32 & 0.883 & 0.867 & 0.934 & 0.922 & 0.984 & 0.977 \\
    64 & 0.888 & 0.873 & 0.940 & 0.928 & 0.985 & 0.978 \\
    128 & 0.884 & 0.868 & 0.941 & 0.929 & 0.986 & 0.979 \\
    \bottomrule
  \end{tabular}
\label{table:fed-batch-coverage}
\end{table}

\subsection{Fed-batch ethanol fermentation under uncertainty}
\label{sec:ethanol-fermentation}

Our second example is based on the fed-batch ethanol fermentation problem in
Case Study~II of \citet{Banga2005}. The original problem is deterministic and
has a free final time. In our implementation, we fix the final time at the value
\(t_f=61.17\)~hours, the value reported for the best solution in
\citet{Banga2005}. Similar to our first example, we introduce parametric
uncertainty in the dynamics. The state is \(x(t,\xi)\in\mathbb R^4\), where
\(x_1\), \(x_2\), and \(x_3\) denote the cell mass, substrate, and ethanol
concentrations, respectively, and \(x_4\) denotes the reactor volume. The scalar
control \(u(t)\) is the feed rate. For a parameter \(\xi\in\mathbb R^5\), the
dynamics are
\[
\begin{aligned}
\dot x_1(t,\xi)
&=
g_1(x(t,\xi),\xi)x_1(t,\xi)
-
\frac{u(t)}{x_4(t)}x_1(t,\xi),
\\
\dot x_2(t,\xi)
&=
-10 g_1(x(t,\xi),\xi)x_1(t,\xi)
+
\frac{u(t)}{x_4(t)}
\bigl(150-x_2(t,\xi)\bigr),
\\
\dot x_3(t,\xi)
&=
g_2(x(t,\xi),\xi)x_1(t,\xi)
-
\frac{u(t)}{x_4(t)}x_3(t,\xi),
\\
\dot x_4(t)
&=
u(t).
\end{aligned}
\]
The kinetic functions are
\[
\begin{aligned}
g_1(x,\xi)
&=
\frac{\xi_1}{1+x_3/\xi_2}
\frac{x_2}{\xi_3+x_2},
\quad \text{and} \quad
g_2(x,\xi)
=
\frac{1}{1+x_3/\xi_4}
\frac{x_2}{\xi_5+x_2}.
\end{aligned}
\]
Here $g_1$ governs cell growth and $g_2$ governs ethanol production.
Both increase with the substrate concentration $x_2$ and decrease with the
ethanol concentration $x_3$, reflecting substrate limitation and product
inhibition.
The initial condition is
$
x_0(\xi)=(1,150,0,10)
$
and we use the control constraints $0\le u(t)\le 12$, $t \in [0, t_f]$, and
\(x_4(t_f)\le 200\). 
The feed supplies substrate at concentration $150$ while increasing the
reactor volume. As substrate is consumed, feeding tends to replenish it through
the term $(u/x_4)(150-x_2)$, while simultaneously diluting the cell-mass and
ethanol concentrations through the terms proportional to $u/x_4$.
The performance index is the negative of the total amount
of ethanol, that is,
\[
F(x(t_f,\xi))
=
-x_3(t_f,\xi)x_4(t_f,\xi),
\]
and the nominal parameter vector is
\[
\bar \xi
=
(0.408,\;16,\;0.22,\;71.5,\;0.44).
\]
We use the parametric uncertainty model from \Cref{sec:luus-reactor} with
$r=0.06$.

We outline why this example satisfies
\Cref{ass:F-Lipschitz,ass:c1,ass:reachable,ass:weakly-compact-feasible-set}.
The reactor volume $x_4(t,\xi)=10+\int_0^t u(s)\,\mathrm{d}s$ is
scenario-independent and gives the feasible set $\mathcal U=\{u\in L^2(0,t_f):
0\le u(t)\le12\ \text{for almost every }t\in[0,t_f],\allowbreak
\int_0^{t_f}u(s)\,\mathrm{d}s\le190\}$, which satisfies
\Cref{ass:weakly-compact-feasible-set}. Following the construction in
\Cref{sec:luus-reactor}, the uncertainty set $\Xi$ is the compact box of
parameter vectors obtained by perturbing each nominal kinetic parameter by a
relative amount of at most $r$. Nagumo-type invariance criteria provide nonnegative concentrations and
$x_2(t,\xi)\le150$. The state equations imply the balance
$(10x_1(t,\xi)+x_2(t,\xi)-150)x_4(t,\xi)=100$, which bounds $x_1(t,\xi)$,
and $\dot x_3(t,\xi)\le x_1(t,\xi)$ bounds $x_3(t,\xi)$. Therefore,
$x(t,\xi)\in\mathcal X=[0,16]\times[0,150]\times[0,16t_f]\times[10,200]$
for all $u\in\mathcal U$, $t\in[0,t_f]$, and $\xi\in\Xi$, as required by \Cref{ass:reachable}.
For \Cref{ass:c1}, the initial conditions $x_0$ are independent of $\xi$,
and positive volume and denominators in $g_1,g_2$ on $\mathcal X\times\Xi$
give smooth dynamics. The terminal loss $F$ is smooth, and hence Lipschitz on
the compact, convex set $\mathcal X$, as required by \Cref{ass:F-Lipschitz}.

\Cref{fig:ethanol-controls} compares the postprocessed nominal and SAA controls.
The postprocessing and the remaining numerical results mirror the fed-batch
reactor study: \Cref{fig:ethanol-clt} illustrates the finite-sample behavior of
\(N^{1/2}(\widehat J_N^*-J^*)\), while \Cref{fig:ethanol-confidence-intervals}
compares the plug-in and subsampling confidence intervals.

\begin{figure}[t]
\centering
\subfloat[Postprocessed nominal control]{%
  \includegraphics[width=0.47\linewidth]{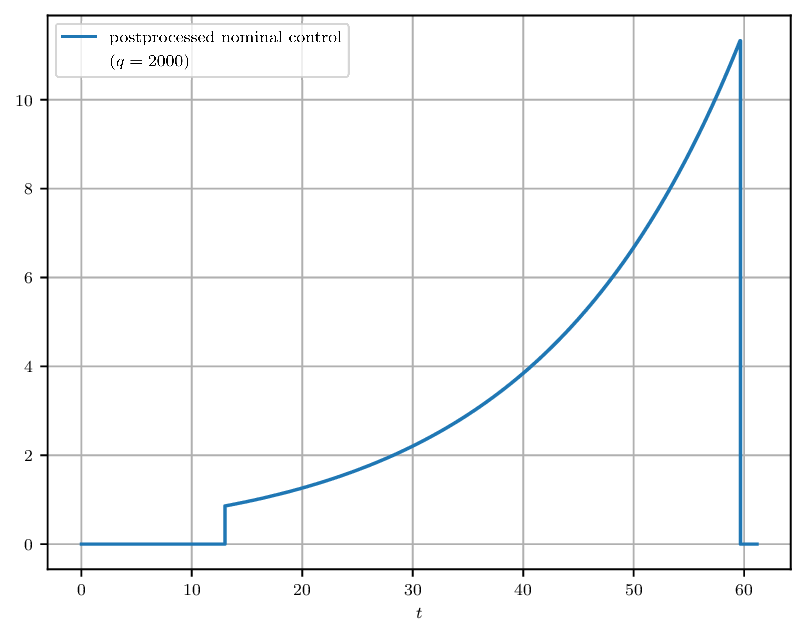}%
}\hfill \subfloat[Postprocessed SAA control]{%
  \includegraphics[width=0.47\linewidth]{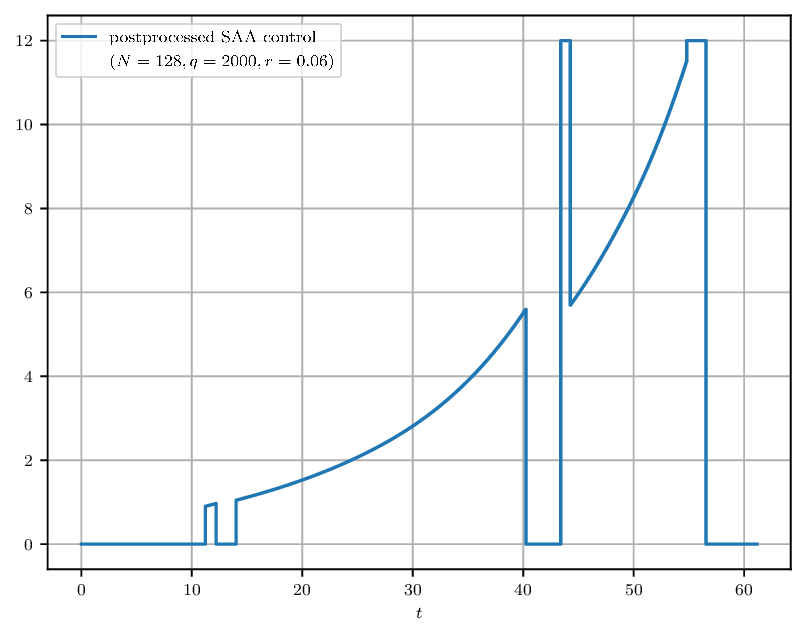}%
} \caption{%
For the ethanol fermentation example, postprocessed nominal and SAA optimal
solutions. }
\label{fig:ethanol-controls}
\end{figure}

\begin{figure}[t]
\centering
\subfloat[Sample size $N = 32$]{%
  \includegraphics[width=0.33\linewidth]{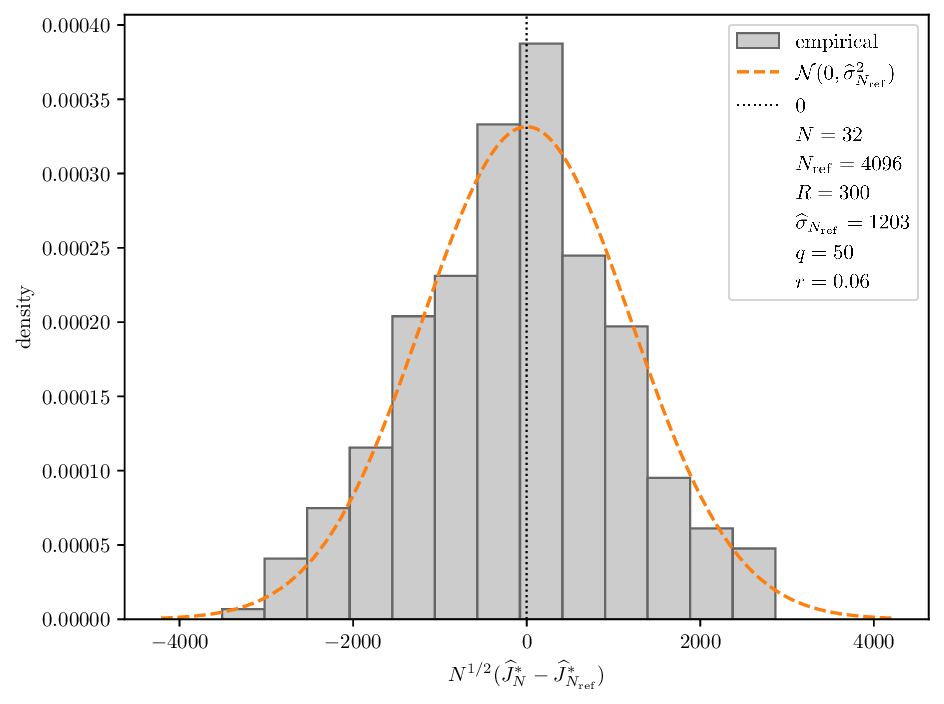}%
}\hfill \subfloat[Sample size $N = 64$]{%
  \includegraphics[width=0.33\linewidth]{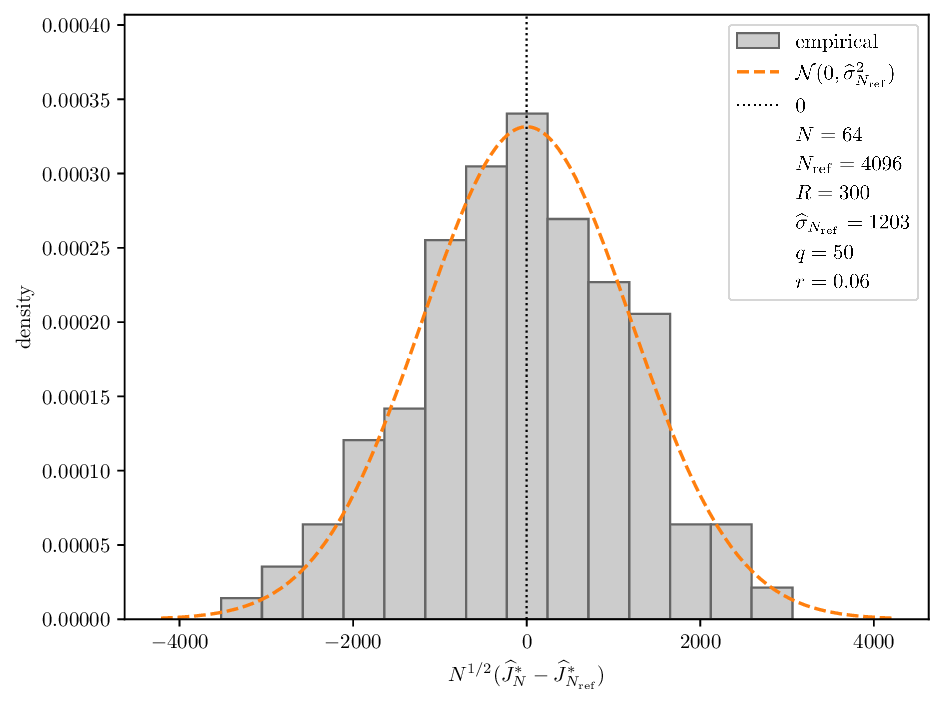}%
}\hfill \subfloat[Optimization bias]{%
  \includegraphics[width=0.33\linewidth]{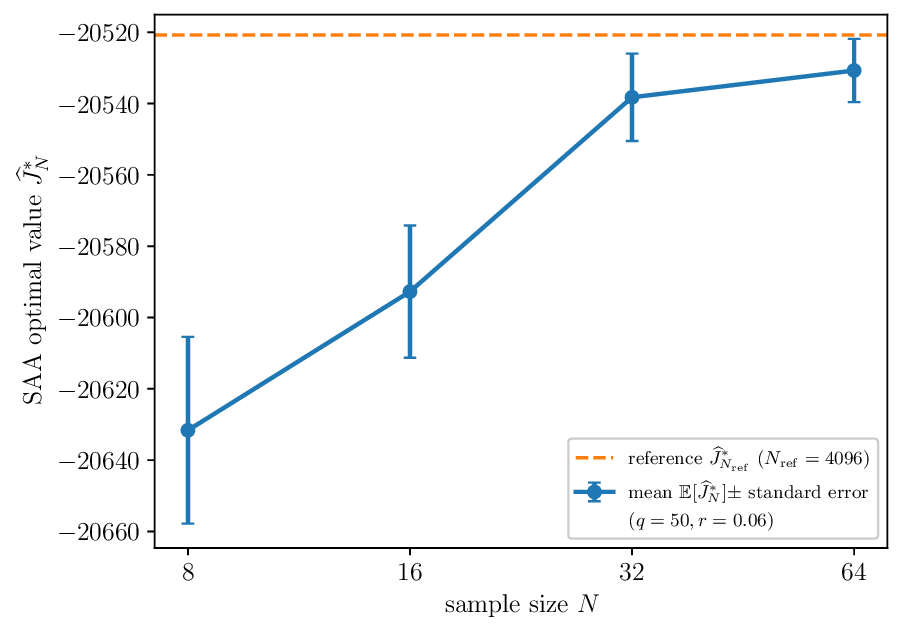}%
} \caption{%
For the ethanol fermentation example, finite-sample behavior of the SAA optimal
value. Panels~(a) and (b) show histograms of \(N^{1/2}(\widehat J_N^*-\widehat
J_{N_\mathrm{ref}}^*)\) over \(R\) i.i.d.\ replications for the indicated sample
size \(N\). Panel~(c) reports the empirical mean of \(\widehat J_N^*\) across
replications. Here \(\widehat J_N^*\) is the SAA optimal value computed from
\(N\) i.i.d.\ parameter scenarios, and \(\widehat J_{N_\mathrm{ref}}^*\) is an
independent reference SAA optimal value computed with \(N_{\mathrm{ref}}\)
sampled parameter realizations. }
\label{fig:ethanol-clt}
\end{figure}

\begin{figure}[!t]
\centering
\subfloat[Plug-in confidence interval]{%
  \includegraphics[width=0.47\linewidth]{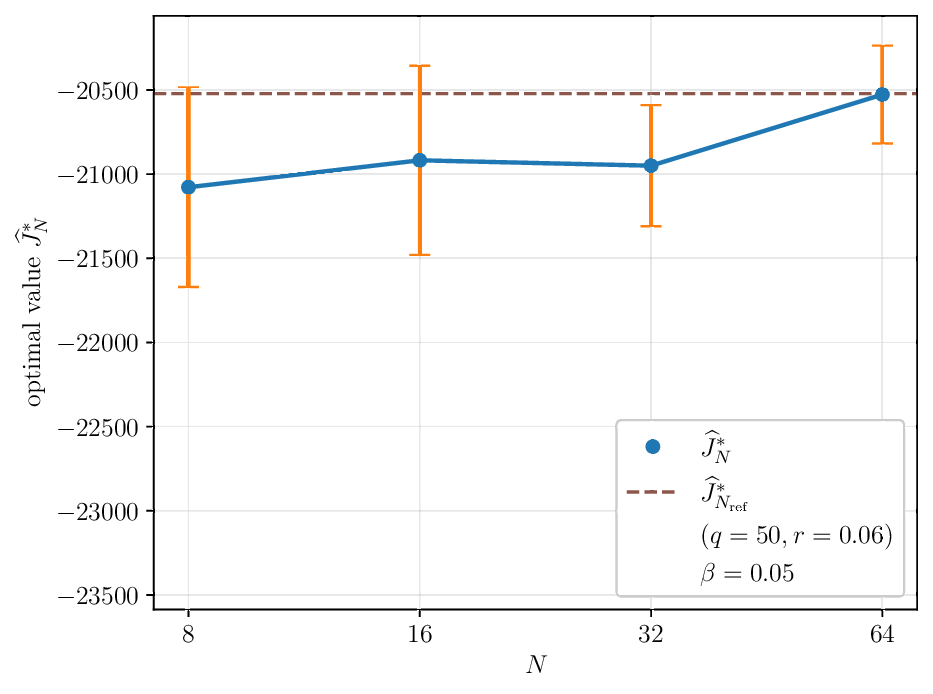}%
}\hfill \subfloat[Subsampling confidence interval]{%
  \includegraphics[width=0.47\linewidth]{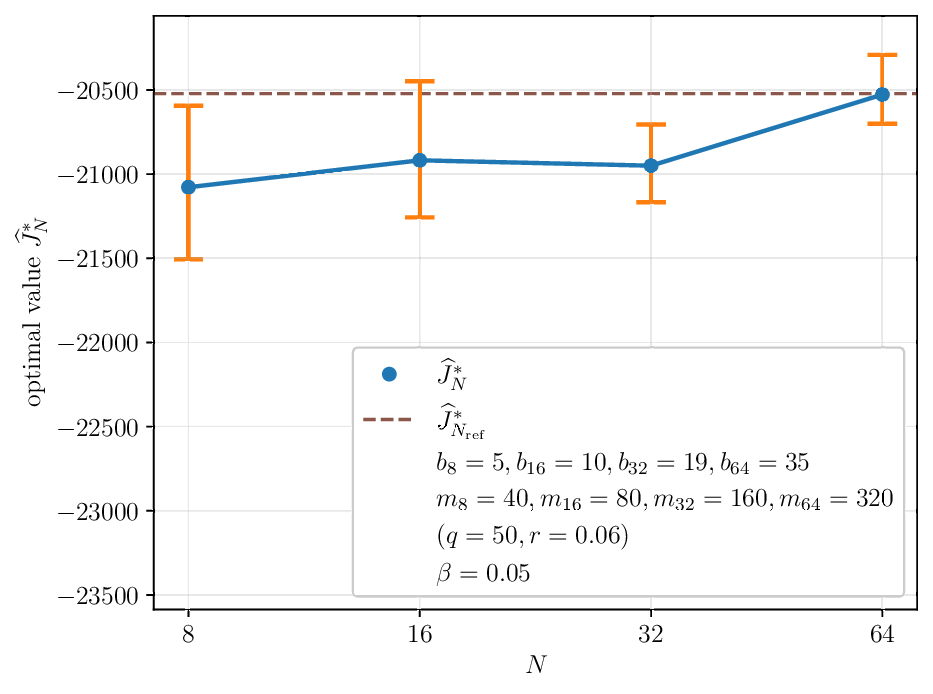}%
} \caption{%
For the ethanol fermentation example, confidence intervals for the population
optimal value \(J^*\). Each panel shows the SAA optimal value \(\widehat J_N^*\)
and its \(1-\beta\) confidence interval for the reported sample sizes. Panel (a)
uses the plug-in normal interval; panel (b) uses the subsampling interval.
The dashed line marks the independent SAA reference value
$\widehat J_{N_{\mathrm{ref}}}^*$ with $N_{\mathrm{ref}}=4096$.}
\label{fig:ethanol-confidence-intervals}
\end{figure}

The coverage results for $N=32$ and $N=64$ are shown in
\Cref{table:ethanol-coverage}. As in \Cref{table:fed-batch-coverage}, we report
the empirical coverage $L/R$ and lower bound $\widehat p_{R,\delta}(L)$ relative
to a fixed SAA reference $\widehat J_{N_{\mathrm{ref}}}^*$ from an independent
sample of $N_{\mathrm{ref}}=4096$ scenarios. Each ethanol SAA solve uses 32
Runge--Kutta steps per control interval, compared with four for the fed-batch
reactor. The additional integration work increases the time per solve, so we
use $R=1000$ replications per sample size, compared with $R=10000$ in
\Cref{table:fed-batch-coverage}. This reduction may explain why the gaps
between empirical coverage $L/R$ and the lower bound $\widehat p_{R,\delta}(L)$
are larger in \Cref{table:ethanol-coverage} than in \Cref{table:fed-batch-coverage}
at the same sample sizes
\citep[see also][Section~10.2.1]{BenTal2009}.

\begin{table}[t]
  \centering
  \caption{For the fed-batch ethanol fermentation example, estimated
  coverage of the plug-in confidence interval for the population optimal
  value, reported as in \Cref{table:fed-batch-coverage}.}
  \label{table:ethanol-coverage}
  \begin{tabular}{rcccccc}
    \toprule
     & \multicolumn{2}{c}{$1-\beta = 0.90$} & \multicolumn{2}{c}{$1-\beta = 0.95$} & \multicolumn{2}{c}{$1-\beta = 0.99$} \\
    \cmidrule(lr){2-3}\cmidrule(lr){4-5}\cmidrule(lr){6-7}
    $N$ & $L/R$ & $\widehat p_{R,10^{-6}}(L)$ & $L/R$ & $\widehat p_{R,10^{-6}}(L)$ & $L/R$ & $\widehat p_{R,10^{-6}}(L)$ \\
    \midrule
    32 & 0.893 & 0.840 & 0.941 & 0.898 & 0.980 & 0.950 \\
    64 & 0.893 & 0.840 & 0.951 & 0.911 & 0.984 & 0.956 \\
    \bottomrule
  \end{tabular}
\end{table}

\section{Discussion}

The justification for the proposed confidence intervals is based on the
stability estimate developed in \Cref{app:stability-estimate}. The estimate
makes it possible to view the SAA objective as a stochastic process indexed by
admissible controls, which leads to a functional central limit theorem for the
SAA objective. The resulting confidence intervals quantify sampling error in SAA
optimal values for a fixed discretization and solution procedure. They indicate
how precisely the population optimal value \(J^*\) is estimated from the sampled
scenarios. A wide interval suggests that the scenario sample is too small to
estimate \(J^*\) accurately, whereas a narrow interval indicates that the
reported optimal value is relatively stable with respect to scenario sampling.
The plug-in interval is computationally inexpensive and is appropriate when the
population optimizer is expected to be unique and the SAA optimizer is stable
across samples. The subsampling interval requires additional SAA solves on
subsamples, but it does not require the population optimizer to be unique. It is
therefore the safer choice when the objective may have several optimal or nearly
optimal controls. The \(N^{-1/2}\) scaling also gives practical guidance:
halving the sampling error requires roughly four times as many scenarios.

Several extensions remain. The present results are restricted to risk-neutral
terminal objectives and open-loop policies. 
For the subsampling confidence interval, the asymptotic theory requires
\(b_N\to\infty\) and \(b_N/N\to0\), but does not determine how \(b_N\)
should be chosen for a finite sample. Developing theory-informed,
practically useful guidelines for selecting the subsample size \(b_N\)
is left for future work.
Future work includes risk-averse
objectives, such as conditional value-at-risk, and state-path constraints
\cite{GeiersbachJuly2025}. Extensions to feedback policies are also important
for batch processes when online measurements can be used during operation
\cite{Bradford2019}. Finally, adaptive sampling rules could be used to decide
sequentially whether additional scenarios are needed \cite{Royset2013}.

\subsection*{Reproducibility of numerical results}

The code and data needed to reproduce the numerical results are archived at
\url{https://doi.org/10.5281/zenodo.24544655}.

\subsection*{Acknowledgment}

During the preparation of this manuscript, the authors used ChatGPT~5.5, 5.6, 6,
and Codex (gpt-6.2-sol) to assist with language editing,
organization, and presentation writing. The authors reviewed and edited all
AI-assisted material and are responsible for the final content.

\subsection*{Funding}
This material is based upon work supported by the U.S.\ Department of Energy,
National Nuclear Security Administration, Advanced Simulation and Computing
Program; the Sandia Laboratory Directed Research and Development Program; and
the National Science Foundation under Award No.\ DMS-2410944.\footnote{%
This article has been authored by an employee of National Technology \&
Engineering Solutions of Sandia, LLC under Contract No. DE-NA0003525 with the
U.S. Department of Energy (DOE). The employee owns all right, title and interest
in and to the article and is solely responsible for its contents. The United
States Government retains and the publisher, by accepting the article for
publication, acknowledges that the United States Government retains a
non-exclusive, paid-up, irrevocable, world-wide license to publish or reproduce
the published form of this article or allow others to do so, for United States
Government purposes. The DOE will provide public access to these results of
federally sponsored research in accordance with the DOE Public Access Plan:
$$\texttt{https://www.energy.gov/downloads/doe-public-access-plan}$$ This paper
describes objective technical results and analysis. Any subjective views or
opinions that might be expressed in the paper do not necessarily represent the
views of the U.S. Department of Energy or the United States Government. }

\appendix

\renewcommand{\thesection}{\Alph{section}}

\section{Stability estimate}
\label[appendix]{app:stability-estimate}

The following estimate is the main analytic ingredient of the paper. It shows
that, for control-affine dynamics, terminal states are stable with respect to
the cumulative effect of the control. This is useful because weakly compact sets
of  controls need not be norm compact, whereas their cumulative input profiles
have stronger compactness properties.

Let
$
U:L^2(0,t_f;\mathbb{R}^s)\to C(0,t_f;\mathbb{R}^s)
$
be the Volterra operator, that is, let
\[
[U(u)](t)=\int_0^t u(s)\,\mathrm{d}s,
\qquad t\in[0,t_f],
\]
where $C(0,t_f;\mathbb{R}^s)$ is the space of $\mathbb{R}^s$-valued continuous
mappings on $[0,t_f]$. Thus \(U(u)\) is the cumulative input profile associated
with the control \(u\). For feasible controls \(u,v\in\mathcal U\), define
\begin{align}
\label{eq:metric-d}
d(u,v)
=
\max_{t\in[0,t_f]} \big\|[U(u)](t)-[U(v)](t)\big\|_2.
\end{align}
Here $\|\cdot\|_2$ denotes either the Euclidean norm or
its induced operator norms.
While not explicitly used by us, 
on the weakly compact set \(\mathcal U\), the metric \(d\) metrizes the
weak topology, which follows from \Cref{lem:primitive_control_compactness}
in \Cref{app:proof-optimal-value-clt}. Thus, convergence
in \(d\) is equivalent to weak convergence of feasible controls.

Our stability estimate shows that this distance controls the difference between
the corresponding terminal states, uniformly over the uncertainty set.
Under our stronger smoothness and compact reachability assumptions,
this estimate gives a quantitative counterpart to the continuous dependence
results of \citet{Liu1999}, which relate trajectory convergence for
control-affine systems to uniform convergence of the cumulative inputs \(U(u)\), 
together with uniform $L^1(0,t_f;\mathbb{R}^s)$-boundedness of the controls.

\begin{proposition}[Endpoint primitive stability]
\label{lem:endpoint_primitive_stability}
Under \Cref{ass:weakly-compact-feasible-set,ass:reachable,ass:c1}, there exists
a constant \(C_{\mathcal U,t_f}<\infty\) such that, for all controls
\(u,v\in\mathcal U\),
\begin{equation}
\label{eq:endpoint_estimate}
\sup_{\xi\in\Xi}\|x_u(t_f,\xi)-x_v(t_f,\xi)\|_2
\le
C_{\mathcal U,t_f}d(u,v).
\end{equation}
\end{proposition}

\begin{proof}
For $j \in \{0, 1\}$, define
\begin{align*}
M_j = \sup_{(x, \xi)\in \mathcal{X} \times \Xi}\, \|f_j(x,\xi)\|_2,
\quad \text{and} \quad 
L_j = \sup_{(x, \xi)\in \mathcal{X} \times \Xi}\, \|D_xf_j(x,\xi)\|_2.
\end{align*}
\Cref{ass:reachable,ass:c1} imply
that these suprema are finite. Since $\mathcal{X}$ is convex, the derivative bounds imply that
for all $\xi \in \Xi$ and $j \in \{0, 1\}$, $f_j$ is Lipschitz continuous with Lipschitz constant
$L_j$.
\Cref{ass:weakly-compact-feasible-set} gives
\begin{align*}
R = \sup_{u \in \mathcal{U}}\, \|u\|_{L^1(0,t_f;\mathbb{R}^s)}
\leq 
t_f^{1/2}\sup_{u \in \mathcal{U}}\, \|u\|_{L^2(0,t_f;\mathbb{R}^s)} < \infty.
\end{align*}
Fix \(u,v\in\mathcal U\) and \(\xi\in\Xi\), and suppress \(\xi\)
in the notation. We define
\[
\Delta(t)=x_u(t)-x_v(t),
\qquad
V(t)=[U(u)](t)-[U(v)](t)
\]
and
\[
B(t)=f_1(x_v(t))V(t),
\qquad
y(t)=\Delta(t)-B(t).
\]
The functions $V$, \(B\), and \(y\) are absolutely continuous. For almost every
\(t \in (0,t_f)\),
\[
\dot\Delta(t)
=
f_0(x_u(t))-f_0(x_v(t))
+
\bigl[f_1(x_u(t))-f_1(x_v(t))\bigr]u(t)
+
f_1(x_v(t))\dot V(t),
\]
and
\[
\dot B(t)
=
D_x f_1(x_v(t))[\dot x_v(t)]V(t)
+
f_1(x_v(t))\dot V(t).
\]
The control-affine structure yields cancellation of the terms involving
$\dot V$, giving
\[
\begin{aligned}
\dot y(t)
=
f_0(x_u(t))-f_0(x_v(t))
+
\bigl[f_1(x_u(t))-f_1(x_v(t))\bigr]u(t)
-
D_x f_1(x_v(t))[\dot x_v(t)]V(t).
\end{aligned}
\]
Using the control-affine dynamics, and the uniform bounds on $f_0$ and
$f_1$, we obtain
\[
\|\dot x_v(t)\|_2
= 
\|f_0(x_v(t))
+
f_1(x_v(t))v(t)\|_2
\le
M_0 + M_1\|v(t)\|_2.
\]
Consequently,
 $
\|B(t)\|_2\le M_1\|V(t)\|_2
$
and
\[
\|D_x f_1(x_v(t))[\dot x_v(t)]V(t)\|_2
\le
L_1(M_0+M_1\|v(t)\|_2)\|V(t)\|_2.
\]
Combining the estimates with the Lipschitz bounds for $f_0$ and $f_1$ yields
\[
\|\dot y(t)\|_2
\le
(L_0+L_1\|u(t)\|_2)\|\Delta(t)\|_2
+
L_1(M_0+M_1\|v(t)\|_2)\|V(t)\|_2.
\]
Define 
\[
\begin{aligned}
a(t)
=L_0+L_1\|u(t)\|_2,
\quad \text{and} \quad 
b(t)
=M_1a(t)+L_1\bigl(M_0+M_1\|v(t)\|_2\bigr).
\end{aligned}
\]
Since \(\Delta=y+B\), we have
\[
\begin{aligned}
\|\dot y(t)\|_2
\le{}
a(t)\|y(t)\|_2+b(t)\|V(t)\|_2.
\end{aligned}
\]
Using $y(0) =0$, and $\max_{t\in[0,t_f]}\|V(t)\|_2 = d(u,v)$, and applying
the nonhomogeneous Gronwall inequality gives
\[
\|y(t)\|_2
\le
\int_0^t
b(s)\|V(s)\|_2
\mathrm{e}^{\left(\int_s^t a(r)\,\mathrm dr\right)}\mathrm ds
\leq 
d(u,v)
\int_0^t
b(s)
\mathrm{e}^{\left(\int_s^t a(r)\,\mathrm dr\right)}\mathrm ds.
\]
Combined with 
\[
\int_0^{t_f}a(t)\,\mathrm dt
\le L_0t_f+L_1R, 
\quad \text{and} \quad 
\int_0^{t_f}b(t)\,\mathrm dt
\le
(L_0M_1+L_1M_0)t_f+2L_1M_1R,
\]
and $\Delta=y+B$, we obtain
\[
\sup_{t\in[0,t_f]}\|\Delta(t)\|_2
\le C_{\mathcal U,t_f}d(u,v),
\]
where 
\begin{equation}
\label{eq:constant-c}
C_{\mathcal U,t_f}
=
M_1+
\bigl[(L_0M_1+L_1M_0)t_f+2L_1M_1R\bigr]
\mathrm{e}^{L_0t_f+L_1R}.
\end{equation}
Taking the supremum over \(\xi\in\Xi\) proves
\eqref{eq:endpoint_estimate}.
\end{proof}

The constant $C_{\mathcal U,t_f}$ in \eqref{eq:constant-c} depends on bounds 
for the dynamics, their state
derivatives, and the admissible controls, as well as the time 
horizon $t_f$. Its exponential factor shows
that it may be large for sensitive dynamics and grow rapidly
with $t_f$.

\section{Proof of the statistical limit theorem}
\label[appendix]{app:proof-optimal-value-clt}

We prepare our proof of \Cref{thm:optimal_value_clt}. We  equip $\mathcal{U}$
with the metric $d$ from \eqref{eq:metric-d} and write
$
K=(\mathcal{U},d)
$.
Denote the Sobolev space of square integrable functions with square integrable
derivatives by
$$H^1(0,t_f;\mathbb{R}^s) = \{u \in L^2(0,t_f;\mathbb{R}^s) \colon \dot u \in
L^2(0,t_f;\mathbb{R}^s)  \}.$$

\begin{lemma}[Compactness of the primitive control space]
\label{lem:primitive_control_compactness}
If \Cref{ass:weakly-compact-feasible-set} holds, then the metric space
$K=(\mathcal{U},d)$ is compact.
\end{lemma}

\begin{proof}
Since the Sobolev embedding from $H^1(0,t_f;\mathbb{R}^s)$ to
$C(0,t_f;\mathbb{R}^s)$ is compact, and the Volterra operator $U$ is bounded and
linear from $L^2(0,t_f;\mathbb{R}^s)$ to $H^1(0,t_f;\mathbb{R}^s)$, $U$ is a
compact linear operator. The operator is also injective. Since
\(L^2(0,t_f;\mathbb{R}^s)\) is reflexive and \(\mathcal U\) is bounded and
weakly closed, \(\mathcal U\) is weakly compact. Compact linear operators are
weakly-to-strongly continuous, so \(U(\mathcal U)\) is compact in
\(C(0,t_f;\mathbb{R}^s)\). Finally, \(U\) is an isometry from \((\mathcal U,d)\)
onto \(U(\mathcal U)\). Hence \((\mathcal U,d)\) is compact.
\end{proof}

We denote the space of continuous functions on $K$ by $C(K)$.
The functional central limit theorem below could alternatively be
established using Donsker results from empirical process theory; see,
e.g., \citet{Vaart2023}. We use a \(C(K)\)-valued central limit theorem,
thereby avoiding the introduction of Donsker classes and related
measurability considerations.

\begin{theorem}[Functional central limit theorem]
\label{cor:control_indexed_clt_CK}
If \Cref{ass:weakly-compact-feasible-set,ass:reachable,ass:c1,ass:F-Lipschitz}
hold, then
\[
N^{1/2}(\widehat{J}_N-J)\rightsquigarrow Z
\quad
\text{in} \quad  C(K),
\]
where \(Z\) is a centered Gaussian process on $\mathcal{U}$ with covariance
\[
\operatorname{Cov}(Z(u),Z(v))
=
\operatorname{Cov}
\bigl(
F(x_u(t_f,\xi)),F(x_v(t_f,\xi))
\bigr),
\quad u, v \in \mathcal{U}.
\]
\end{theorem}

\begin{proof}
We apply the functional central limit theorem in Theorem~1 of  \cite{Jain1975}
(see also Theorem~14.2 in \cite{Ledoux1991}) to the mapping
\[
[X(\xi)](u)=F(x_u(t_f,\xi))-J(u).
\]

\Cref{lem:primitive_control_compactness} ensures that $K$ is compact.
\Cref{ass:reachable,ass:F-Lipschitz} ensure that $X$ is essentially bounded.
\Cref{ass:F-Lipschitz} ensures that $F$ is Lipschitz continuous with Lipschitz
constant, say $L_F$. Combined with \Cref{lem:endpoint_primitive_stability}, we
obtain the Lipschitz bound
\begin{align}
\label{eq:composite-lipschitz}
|F(x_u(t_f,\xi))-F(x_v(t_f,\xi))|
\le
L_FC_{\mathcal U,t_f} d(u,v),
\qquad
u,v\in\mathcal U,\ \xi\in\Xi .
\end{align}
In particular, we have
\[
|[X(\xi)](u)-[X(\xi)](v)|\le 2L_FC_{\mathcal U,t_f}d(u,v),
\qquad u,v\in \mathcal U, \ \xi\in\Xi.
\]

Using standard arguments (cf., e.g., Example~1.5.1 in \cite{Vaart2023}), it can
be shown that $X$ is a $C(K)$-valued random element. It is also mean zero, as we
subtract the objective $J$. Since the $\xi_i$ are i.i.d.\ random vectors,
$X(\xi_i)$ are i.i.d.\ $C(K)$-valued random elements.

We now verify the entropy condition in Theorem~1 in \cite{Jain1975}, which
requires us to introduce covering numbers. For a metric space $(S,\rho)$, the
covering number $\mathscr{N}(r,S,\rho)$ is the smallest number of balls with
radius $r$, under the metric $\rho$, needed to cover all of $S$. By
\Cref{ass:weakly-compact-feasible-set}, there exists a constant $R_{\mathcal U}
> 0$ such that
\[
U(\mathcal{U}) \subset \mathcal{V} = \big\{
v\in H^1(0,t_f;\mathbb{R}^s):
v(0)=0,\ \|\dot{v}\|_{L^2(0,t_f;\mathbb{R}^s)}\le R_{\mathcal U}
\big\}.
\]
The set $\mathcal{V}$ is bounded in $H^1(0,t_f;\mathbb{R}^s)$. Therefore, the
Sobolev entropy estimate in the corollary on p.~315 in \cite{Birman1967} gives,
for all sufficiently small \(\varepsilon>0\),
\[
\log \mathscr{N}(\varepsilon,\mathcal{V},\|\cdot\|_{C(0,t_f;\mathbb{R}^s)})\lesssim \varepsilon^{-1}.
\]
Combined with the fact that $U$ is an isometry, we obtain, for all sufficiently
small \(\varepsilon>0\),
\[
\log \mathscr{N}(\varepsilon,K,d)\lesssim\varepsilon^{-1}.
\]
Consequently,
\[
\int_0^{1}
\sqrt{\log \mathscr{N}(\varepsilon,K,d)}\,\mathrm{d}\varepsilon
<\infty .
\]
Now, Theorem~1 in \cite{Jain1975} implies the assertions.
\end{proof}

\begin{proof}[{Proof of \Cref{thm:optimal_value_clt}}]
By \Cref{lem:primitive_control_compactness,cor:control_indexed_clt_CK}, the
metric space \(K=(\mathcal U,d)\) is compact and
\(N^{1/2}(\widehat{J}_N-J)\rightsquigarrow Z\) in \(C(K)\), so the arguments in
the proof of Theorem~5.7 in \cite{Shapiro2021} apply verbatim.
\end{proof}

\section{Proof of confidence interval schemes}
\label[appendix]{app:proof-confidence-intervals}

\begin{proof}[{Proof of \Cref{thm:gaussian_ci_unique_optimizer}}]
We define $\sigma^2 = \operatorname{Var}(F(x_{u^*}(t_f,\xi)))$. We first show
that, with probability one, $\widehat{\sigma}_N^2\to\sigma^2$. Using the
Lipschitz bound \eqref{eq:composite-lipschitz}, the fact that $K = (\mathcal{U},
d)$ is compact, and the uniform law of large numbers (see Corollary 4:1 in
\cite{LeCam1953}), we obtain, with probability one,
\[
\sup_{u\in\mathcal U}|\widehat{J}_N(u)-J(u)|\to0.
\]
Since \(J\) is continuous on compact \((\mathcal U,d)\) and has the unique
minimizer \(u^*\), the arguments from the proof of Theorem~5.3 in
\cite{Shapiro2021} imply that, with probability one, \(\widehat{J}_N^*\to J^*\)
and \(d(u_N^*,u^*)\to0\).

Next, since \(x_u(t_f,\xi)\in\mathcal X\) and \(F\) is continuous on the compact
set \(\mathcal X\), \(B_F=\sup_{x\in\mathcal X}|F(x)|<\infty\). Combined with
the Lipschitz bound \eqref{eq:composite-lipschitz}, with probability one,
\[
\begin{aligned}
\bigg|
\frac1N\sum_{i=1}^N F(x_{u_N^*}(t_f,\xi_i))^2
-
\frac1N\sum_{i=1}^N F(x_{u^*}(t_f,\xi_i))^2
\bigg|
\le
2B_F L_FC_{\mathcal U,t_f}  d(u_N^*,u^*)
\to0.
\end{aligned}
\]
The strong law of large numbers gives, with probability one,
$
(1/N)\sum_{i=1}^N F(x_{u^*}(t_f,\xi_i))^2
\to
\mathbb E[F(x_{u^*}(t_f,\xi))^2]
$.
Hence, with probability one,
\[
\frac1N\sum_{i=1}^N F(x_{u_N^*}(t_f,\xi_i))^2
\to
\mathbb E[F(x_{u^*}(t_f,\xi))^2].
\]
Finally, with probability one,
$
\widehat{\sigma}_N^2
\to
\mathbb E[F(x_{u^*}(t_f,\xi))^2]
-
\mathbb E[F(x_{u^*}(t_f,\xi))]^2
=
\sigma^2
$.

Since \(\mathcal U^*=\{u^*\}\), \Cref{thm:optimal_value_clt} gives
$
N^{1/2}(\widehat{J}_N^*-J^*)
\rightsquigarrow
\mathcal{N}(0,\sigma^2)
$.
Since \(\widehat{\sigma}_N\to\sigma>0\) in probability, Slutsky's theorem gives
$\widehat\sigma_N^{-1}N^{1/2}(\widehat{J}_N^*-J^*) \rightsquigarrow
\mathcal{N}(0,1)$. The coverage follows from continuity of the standard normal
distribution at \(\pm z_{1-\beta/2}\).
\end{proof}

\begin{proof}[Proof of \Cref{thm:subsampling_ci}]
By \Cref{thm:optimal_value_clt},
$
N^{1/2}(\widehat J_N^*-J^*)\rightsquigarrow \inf_{u\in\mathcal U^*} Z(u)
$.
Since \(b_N\to\infty\), \(b_N/N\to0\), and \(m_N\to\infty\), Corollary~2.1 in
\cite{Politis1994}, in the stochastic programming form of Section~5.3 in
\cite{Eichhorn2007}, implies that the empirical law of
$
b_N^{1/2}(\widehat J_{I_r}^*-\widehat J_N^*)
$
consistently estimates the law of \(\inf_{u\in\mathcal U^*} Z(u)\) at continuity
points. 
Since, by assumption, this limiting law has no atom at its lower
\(\beta/2\)- and \((1-\beta/2)\)-quantiles, 
Corollary~2.1 of \cite{Politis1994}, applied to the two tails, yields
$
\lim_{N\to\infty}
\mathrm{Prob}
\bigl\{
J^*\in\mathcal I_{N,b_N,m_N}^{\mathrm{sub}}(\beta)
\bigr\}
= 1-\beta
$.
\end{proof}

\section{Coverage of a Student-$t$ plug-in interval}
\label[appendix]{app:student-t-coverage}

In addition to the coverage analysis for the plug-in confidence intervals
in \Cref{table:fed-batch-coverage,table:ethanol-coverage},
we  assess a finite-sample modification of the plug-in confidence
interval. Using the variance estimator $\widehat\sigma_N^2$ from 
\Cref{alg:plugin_ci_unique_optimizer}, we define
\begin{equation}
\label{eq:student-t-plug-in}
\mathcal{I}_N^t=
\Big[
\widehat J_N^*
- 
t_{N-1,\,1-\beta/2}\frac{\widehat{s}_N}{N^{1/2}}, 
\quad
\widehat J_N^*
+
t_{N-1,\,1-\beta/2}\frac{\widehat{s}_N}{N^{1/2}}
\Big],
\end{equation}
where
$\widehat{s}_N^2=(N/(N-1))\widehat\sigma_N^2$, 
and $t_{N-1,\,1-\beta/2}$ is the $(1-\beta/2)$-quantile of the
Student-$t$ distribution with $N-1$ degrees of freedom.
This modification preserves the asymptotic coverage guarantee of
\Cref{thm:gaussian_ci_unique_optimizer}, since
$t_{N-1,\,1-\beta/2}\to z_{1-\beta/2}$ as $N\to\infty$.

For the fed-batch reactor from \Cref{sec:luus-reactor}, we use the same
replications, fixed reference value, and coverage estimators as in
\Cref{table:fed-batch-coverage}.
The results in \Cref{table:fed-batch-coverage-t} show coverage closer
to the nominal levels than for the normal plug-in interval,
particularly for small $N$, although undercoverage remains in
several settings.
\Cref{table:ethanol-coverage-t} gives the corresponding results for
fed-batch ethanol fermentation from \Cref{sec:ethanol-fermentation}.

\begin{table}[t]
  \centering
\caption{For the fed-batch reactor example (see \Cref{sec:luus-reactor}),
estimated coverage of the Student-$t$ plug-in confidence
interval from \eqref{eq:student-t-plug-in}, using the same
replications, reference value, and reporting conventions as in
\Cref{table:fed-batch-coverage}.}
  \begin{tabular}{rcccccc}
    \toprule
     & \multicolumn{2}{c}{$1-\beta = 0.90$} & \multicolumn{2}{c}{$1-\beta = 0.95$} & \multicolumn{2}{c}{$1-\beta = 0.99$} \\
    \cmidrule(lr){2-3}\cmidrule(lr){4-5}\cmidrule(lr){6-7}
    $N$ & $L/R$ & $\widehat p_{R,10^{-6}}(L)$ & $L/R$ & $\widehat p_{R,10^{-6}}(L)$ & $L/R$ & $\widehat p_{R,10^{-6}}(L)$ \\
    \midrule
    8 & 0.904 & 0.889 & 0.950 & 0.939 & 0.989 & 0.984 \\
    16 & 0.900 & 0.886 & 0.948 & 0.937 & 0.988 & 0.982 \\
    32 & 0.898 & 0.883 & 0.950 & 0.939 & 0.990 & 0.984 \\
    64 & 0.896 & 0.880 & 0.947 & 0.935 & 0.989 & 0.983 \\
    128 & 0.887 & 0.871 & 0.944 & 0.932 & 0.988 & 0.982 \\
    \bottomrule
  \end{tabular}
  \label{table:fed-batch-coverage-t}
\end{table}

\begin{table}[t]
  \centering
  \caption{For the fed-batch ethanol fermentation example (see
  \Cref{sec:ethanol-fermentation}), estimated coverage of the Student-$t$
  plug-in interval \eqref{eq:student-t-plug-in}, using the same replications, reference
  value, and reporting conventions as in
  \Cref{table:ethanol-coverage}.}
  \label{table:ethanol-coverage-t}
  \begin{tabular}{rcccccc}
    \toprule
     & \multicolumn{2}{c}{$1-\beta = 0.90$} & \multicolumn{2}{c}{$1-\beta = 0.95$} & \multicolumn{2}{c}{$1-\beta = 0.99$} \\
    \cmidrule(lr){2-3}\cmidrule(lr){4-5}\cmidrule(lr){6-7}
    $N$ & $L/R$ & $\widehat p_{R,10^{-6}}(L)$ & $L/R$ & $\widehat p_{R,10^{-6}}(L)$ & $L/R$ & $\widehat p_{R,10^{-6}}(L)$ \\
    \midrule
    32 & 0.906 & 0.855 & 0.954 & 0.915 & 0.993 & 0.971 \\
    64 & 0.902 & 0.851 & 0.958 & 0.920 & 0.989 & 0.964 \\
    \bottomrule
  \end{tabular}
\end{table}

\bibliographystyle{elsarticle-harv}
\bibliography{references}

\end{document}